\documentclass[10pt,twoside, a4paper, english, reqno]{amsart}
\usepackage[dvips]{epsfig}
\usepackage{amscd}
\usepackage{stmaryrd}
\usepackage{amssymb}
\usepackage{amsthm}
\usepackage{amsmath}
\usepackage{graphicx,enumitem,dsfont}
\usepackage{latexsym}

\usepackage{upref}
\usepackage[colorlinks]{hyperref}
\usepackage{color}
\usepackage{subcaption}
\usepackage[usenames,dvipsnames]{xcolor}
\theoremstyle{plain}
\newtheorem{thm}{Theorem}[section]
\theoremstyle{plain}
\newtheorem{lem}[thm]{Lemma}
\newtheorem{prop}[thm]{Proposition}

\theoremstyle{definition}
\newtheorem{defi}{Definition}[section]
\newtheorem{rem}{Remark}[section]

\newtheorem*{maintheorem*}{Main Theorem}

\newenvironment{Assumptions}
{
\setcounter{enumi}{0}

\begin{enumerate}}
{\end{enumerate} }

\newcommand{\R}{\ensuremath{\mathbb{R}}}

\newcommand{\goto}{\ensuremath{\rightarrow}}
\newcommand{\grad}{\ensuremath{\nabla}}
\newcommand{\eps}{\ensuremath{\varepsilon}}

\numberwithin{equation}{section} \allowdisplaybreaks

\title[Renormalized entropy solution for degenerate SPDEs]
{Renormalized stochastic  entropy solution for degenerate parabolic-hyperbolic equations with L\'evy noise}

\date{}

\author[ S. R. Behera ]{Soumya Ranjan Behera}
\address[Soumya Ranjan Behera] {\newline 
Department of Mathematics,
Indian Institute of Technology Delhi,
Hauz Khas, New Delhi, 110016, India.}
\email[] {maz198759@iitd.ac.in}

\author[A. K. Majee]{Ananta K. Majee}
\address[Ananta K. Majee]{\newline
Department of Mathematics,
Indian Institute of Technology Delhi,
Hauz Khas, New Delhi, 110016, India. }
\email[]{majee@maths.iitd.ac.in}


\thanks{}
\begin{document}
\begin{abstract}
  In this article, we establish the well-posedness theory for renormalized entropy solutions of a degenerate parabolic-hyperbolic PDE perturbed by a multiplicative Levy noise with general $L^1$-data on $\R^d$. By using a suitable approximation procedure based on the vanishing viscosity technique and bounded data, we prove the existence of a renormalized entropy solution to the underlying problem. The uniqueness of the solution is settled by adapting Kru\v{z}kov's doubling the variables technique in the presence of noise. 
\end{abstract}
\maketitle
\section{Introduction}
Let $(\Omega, \mathcal{F}, \mathcal{P}, (\mathcal{F}_t)_{t \in [0,T]})$ be a stochastic basis and $T >0$. We are interested in the Cauchy problem for a nonlinear degenerate parabolic-hyperbolic stochastic PDE of the  type:
\begin{align} \label{eq:stoc}
\begin{cases} 
\displaystyle du(t,x)- \mbox{div} f(u) \,dt-\Delta \Phi(u)\,dt =
\sigma(u)\,dW(t) + \int_{\mathbf{E}} \eta(u;z)\widetilde{N}(dz,dt),~~(x,t) \in \Pi_T, \\
u(0,x) = u_0(x),~~ x\in \R^d,
\end{cases}
\end{align}
where $\Pi_T:= (0,T)\times \R^d$ with $T>0$ fixed, $u_0:\R^d \mapsto \R$ is the given integrable initial
function, $f:\R \mapsto \R^d$ is a given (sufficiently smooth)  flux function 
(see Section~\ref{sec:technical} for the complete list of assumptions), and $\Phi:\R \mapsto \R$ is a given nonlinear function. Regarding this, the basic assumption is that $\Phi(\cdot)$ is nondecreasing 
with $\Phi(0)=0$. Furthermore, $W(t)$ is a one dimensional Brownian noise, and $ \widetilde{N}(dz,dt)= N(dz,dt)-m(dz)\,dt $, 
where $N$ is a Poisson random measure on $(\mathbf{E},\mathcal{E})$ with intensity measure $m(dz)$, where $(\mathbf{E},\mathcal{E},m)$ is a $\sigma$-finite measure space.
Finally, $ u \mapsto \sigma (u)$ and $(u,z)\mapsto \eta(u,z)$ are given real valued functions signifying
the multiplicative nature of the noise.
\vspace{0.1cm}

The equation of the type \eqref{eq:stoc} arises in many different fields where non-Gaussianity plays an important role. For example, it is used to model the phenomenon of two or three-phase flows in porous media \cite{Karlsen-2000-LN} or sedimentation-consilation processes \cite{Burger-1999}. 
 In the absence of noise terms, \eqref{eq:stoc} becomes a deterministic degenerate parabolic–hyperbolic equation in $\R^d$. A number of authors have contributed since then, and we mention few of the works e.g, \cite{Maliki-2010, Carrillo-1999, kal-resibro,cockburn,Evje-2001,Wu-1989,Hudjaev-1969} and
references therein for its entropy solution theory at least when the data are regular enough (say $L^1\cap L^\infty$). For purely $L^1$-data, one cannot expect a solution~(of deterministic degenerate parabolic–hyperbolic equation) with more than $L^1$-regularity and therefore, in general, it is not possible to consider the solution of the underlying PDEs in the distributional sense. The notion of renormalized entropy solution, first introduced  by DiPerna and Lions \cite{Diperna-renormalized} for the study of the Boltzmann equation, was adapted by many authors to study the underlying deterministic problems for $L^1$-data; see for example
 \cite{Karlsen-2004,Wittbold-2015-renormalized,Wittbold-2002-renormalized,Wittbold-2000-renormalized,Chen-2003-renormalized,Blanchard-2001} and references therein. 
\vspace{0.1cm}

The study of stochastic degenerate parabolic-hyperbolic equations has so far been limited to mainly equations with a Brownian noise. In fact, Kim \cite{Kim-2003} extended the Kru\v{z}kov well-posedness theory to one-dimensional balance laws that are driven by additive Brownian noise. For multiplicative noise, one cannot apply straightforward Kru\v{z}kuv's doubling the variables technique to get uniqueness. This issue was settled by Feng and Nualart \cite{nualart:2008}, who
established the uniqueness of the entropy solution by recovering additional information from the vanishing viscosity method. The existence was proven in one spatial dimension by using a stochastic version
of compensated compactness method; see also \cite{BisMaj}.  Debussche and Vovelle \cite{Vovelle2010} introduced kinetic formulation
of such problems and established the well-posedness theory for
multidimensional stochastic balance law via kinetic approach. By introducing the notion of generalized stochastic entropy solution, Bauzet et al. \cite{BaVaWit_2012} established the well-posedness of entropy solution for multi-dimensional balance laws driven by multiplicative noise via the Young measure technique. 
In \cite{Karlsen-2017}, the authors slightly modified the Kru\v{z}kov's entropy inequality, allowing the Kru\v{z}kov's constants to be Malliavin differentiable random variables and proved well-posedness theory for entropy solution---which simplifies some of the proofs of \cite{nualart:2008,BaVaWit_2012}. We mention the works of Chen et al. \cite{Chen-karlsen_2012}, where well-posedness of entropy solution is established in $L^p\cap {\rm BV}$, via BV framework. See also the works of  Dotti et al. \cite{Vovelle-2018}, Vallet et al. \cite{Vallet_2009}, and Bauzet et al. \cite{BaVaWit_JFA}.  Extending the work of \cite{BaVaWit_2012}, Bauzet et al. \cite{BaVaWit_2014} established the existence and uniqueness of solutions of \eqref{eq:stoc} in the case $\eta=0$. By using a vanishing viscosity method, based on the compactness
proposed by the theory of Young measures, they have proved the existence of
entropy solution.  The uniqueness of the solution is obtained via Kruzkov’s doubling variable method. Adapting the notion of kinetic
formulation,  Debussche et al. \cite{martina-2016} developed a well-posedness theory for \eqref{eq:stoc} with $\eta=0$; see also \cite{Martina-2013, Chaudhary-2023}. We also mention the works of  Bhauryal et al. \cite{Koley-2021}, where they have established well-posednes theory and developed a continuous dependence theory in stochastic entropy solution framework for fractional degenerate parabolic-hyperbolic Cauchy problem driven by Brownian noise. 
Very recently, Behera and Majee \cite{Behera-CLT} established the Freidlin-Wentzell type large deviation
principle and central limit theorem for stochastic fractional conservation laws with small multiplicative noise in kinetic formulation framework via weak convergence approach. 

\vspace{0.1cm}

Being relatively new area of pursuit, the study of stochastic scalar balance laws driven by L\'{e}vy noise is even more sparse than the Brownian noise case.
In recent years, Biswas et al. \cite{BisMajKarl_2014, BisMajVal-2019, BisKoleyMaj}, Koley et al. \cite{KoleyMajVa-2017} established the well-posedness theory of the entropy solution within $L^p$-solution framework,  for \eqref{eq:stoc}  with a Poisson noise via the Young measure approach. In \cite{BisKoleyMaj, KoleyMajVa-2017}, the authors also developed a continuous dependence theory on nonlinearities within the BV solution setting. We also mention the work of Bhauryal et al. \cite{Koley-2020}, where they established existence, uniqueness, and continuous dependence theory for multidimensional fractional conservation laws driven by L\'evy noise in the entropy solution framework. Moreover, they established a result on vanishing non-local regularization of scalar stochastic conservation laws. Very recently, the authors in \cite{Behera-boubded-domain} have studied the homogeneous Dirichlet problem for a degenerate parabolic-hyperbolic PDE perturbed by L\'evy noise and established the well-posedness theory of entropy solution via  Kru\v{z}kov’s semi-entropy formulation. We also refer to \cite{SRB-24, Majee-2018, Majee-2018-flux} for various numerical studies of conservation law driven by L\'evy noise. 

\vspace{0.1cm}

However, the study of well-posedness theory (in entropy solution framework) for stochastic balance laws driven by noise for purely $L^1$-data is rare.  Within the existing literature, we can refer to the paper by Lv et al. \cite{Lv-2016-renormalized}, where the authors wanted to establish the existence and uniqueness of the renormalized stochastic entropy solutions of non-homogeneous Dirichlet problem on stochastic scalar conservation law driven by a multiplicative Brownian noise for a general $L^1$-data. 
In this article, taking primary motivation from \cite{Karlsen-2004, Blanchard-2001}, we wish to offer an alternative “pure” $L^1$ well-posedness theory for \eqref{eq:stoc} with purely integrable initial data, based on the notion of renormalized stochastic entropy solutions and the adaptation of classical Kru\v{z}kov's \cite{Kruzkov-1970} method in stochastic setting.  Following the deterministic counterpart \cite{Karlsen-2004}, we first define the notion of renormalized stochastic entropy solution and prove its well-posedness. 
\begin{itemize}
\item[i)] As a first step, we first derive a-priori estimates for viscous solution $u_\eps$ (with regular initial data) of \eqref{eq:viscous-stoc} and its associated Kirchhoff's term involving the Lipschitz continuous truncation function ${\tt T}_{\ell}$ at height $\ell>0$ defined in \eqref{eq:trucn-fn}. Moreover, by identifying the appropriate random Radon measure on $\Pi_T$, and with the help of strong convergence of $u_\eps$ to $u$ in $L^p(\Omega\times(0,T); L^p_{\rm loc}(\R^d))$ for $1\le p<2$ ~(cf.~Theorem \ref{thm:L1-contraction-entropy}), we have shown that $u$ is a renormalized stochastic entropy solution which then turns out, by the uniqueness of entropy solution, that any entropy solution of the underlying problem for initial data $u_0 \in L^1(\R^d)\cap L^\infty(\R^d)$ is indeed a renormalized stochastic entropy solution. To do so, we take a special function $S=S_{\beta, {\tt h}_{\ell, \delta}}$ defined by \eqref{eq:special-fun-S} in the generalized It\^o-l\'evy formula \eqref{eq:ito-formula-viscous} on $u_\eps$ and then pass to the limit as $\delta \goto 0$ and subsequently $\eps\goto 0$. An important ingredient is to pass to the
limit in the martingale term $\mathcal{H}_5$ and the It\^o correction term $\mathcal{H}_7$, which are non-local in nature, coming from the jump noise. Moreover, due to the nonlinear nature of the truncation function ${\tt T}_{\ell}$, one cannot get the desired inequality in a straightforward manner after passing to the limit in these terms. We first take their sum and re-write it into the desired inequality and the associated error terms. With the help of the nondecreasing property of $\eta$ in its first argument, $\eta(0,z)=0$ for all $z\in \mathbf{E}$ together with the fact that $\beta^\prime(0)=0$, and the definition of compensated Poisson random measure, we are able to show that error terms are non-positive, cf.~\eqref{inq:i2j2}.
\item [ii)] By using the analysis of ${\rm i)}$ and the boundedness property of the noise coefficiets i.e., the assumption \ref{A7}~(which is needed to show the dissipation condition of the renormalized measure),~we have shown that the point-wise and $L^1(\Omega \times \Pi_T)$ limit process $\bar{u}$ of the approximate sequence $\{u_n\}_{n\in \mathbb{N}}$, the unique entropy solution to the problem
\eqref{eq:stoch-truncation}, is indeed a renormalized stochastic entropy solutions of \eqref{eq:stoc} for purely $L^1$-initial data.  
\item[iii)]  For uniqueness of renormalized solutions, we follow the methodology of \cite{BisMajKarl_2014, BisMajVal-2019,BaVaWit_2014} and derive a Kato type inequality \eqref{inq:3-uni} by comparing any renormalized solution with the viscous solution based on the stochastic version of Kru\v{z}kov's doubling the variables technique. Since the viscous limit is indeed the unique entropy solution,  by using the Kato inequality for the entropy solution $u_n$, we prove that any renormalized stochastic entropy solution of the underlying problem is the  $L^1$-limit of the sequence $\{u_n\}$. 
\end{itemize}

\vspace{0.1cm} 

The remaining part of this paper is organized as follows: In Section \ref{sec:technical}, we introduce the notion of a renormalized stochastic entropy solution
 for \eqref{eq:stoc},  state the assumptions, detail of the technical
framework, and state our main well-posedness result.  Section \ref{sec:existence} is devoted to the proof of the existence of a renormalized stochastic entropy solution of \eqref{eq:stoc}. In the final section, we complete the uniqueness of the renormalized stochastic entropy solution. 

\section{Technical Framework and statement of the main results}\label{sec:technical} 
Throughout this paper, we use the letter $C$ to denote various generic constant. There are
situations where constants may change from line to line, but the notation is kept unchanged so long as it
does not impact the central idea. For any separable Hilbert Space $H$, we denote by $N_w^2((0,T); H)$ the space of all square integrable $\{\mathcal{F}_t\}$ predictable $H$-valued process $u$ such that $\displaystyle\mathbb{E}\Big[ \int_0^T \|u(t)\|_H^2\,dt\Big] < \infty.$ 

\vspace{0.1cm}

It is well known that weak solutions may be discontinuous and they are not uniquely determined by their initial data. Consequently, an entropy condition must be imposed to single out the physically correct solution.  To define the notion of stochastic entropy solution, we first recall the definition of entropy flux triplet and  Kirchoff's function associated to $\Phi$. 
\begin{defi}[Entropy flux triple]
 	A triplet $(\beta,\zeta,\nu) $ is called an entropy flux triple if $\beta \in C^2(\R) $ and $\beta \ge0$,
 	$\zeta = (\zeta_1,\zeta_2,....\zeta_d):\R \mapsto \R^d$ is a vector valued function, and $ \nu :\R \mapsto \R $ is a scalar valued function such that 
 	\[\zeta'(r) = \beta'(r)f'(r) \quad \text{and}\quad \nu^\prime(r)= \beta'(r)\Phi'(r).\]
 An entropy flux triple $(\beta,\zeta,\nu)$ is called convex if $ \beta^{\prime\prime}(s) \ge 0,~~\forall~s\in \R$.  
\end{defi}
The associated Kirchoff's function of $\Phi$, denoted by 
${\tt G}(x)$, is defined by
 $${\tt G}(x)= \int_0^x \sqrt{\Phi^\prime(r)}\,dr,~~x\in \R\,.$$ 
With the help of  convex entropy flux triplet $(\beta,\zeta,\nu)$, the notion of stochastic entropy solution of \eqref{eq:stoc} is defined as follows~cf.~\cite{BaVaWit_2014,BisMajVal-2019}.

 \begin{defi} [Stochastic entropy solution]
 \label{defi:stochentropsol}
A $ L^2(\R^d )$-valued $\{\mathcal{F}_t: t\geq 0 \}$-predictable stochastic process $u(t)= u(t,x)$ is called a stochastic entropy solution of \eqref{eq:stoc} if
 \begin{itemize}
 \item[(i)] for each $ T>0$, 
 \begin{align*}
 {\tt G}(u) \in L^2((0,T)\times \Omega;H^1(\R^d)), \,\, \text{and} \,\, 
 \underset{0\leq t\leq T}\sup  \mathbb{E}\big[||u(t,\cdot)||_{L^2(\R^d)}^{2}\big] <+ \infty\,;
 \end{align*}
 \item[(ii)] given a non-negative test function  $\psi\in C_{c}^{1,2}([0,\infty )\times\R^d) $ and a convex entropy flux triple $(\beta,\zeta,\nu)$, the following inequality holds:
 \begin{align*}
 &  \int_{\Pi_T} \Big\{ \beta(u(t,x)) \partial_t\psi(t,x) +  \nu(u(t,x))\Delta \psi(t,x) -  \grad \psi(t,x)\cdot \zeta(u(t,x)) \Big\}dx\,dt \notag \\
 & + \int_{\Pi_T} \sigma(u(t,x))\beta^\prime (u(t,x))\psi(t,x)\,dW(t)\,dx
 + \frac{1}{2}\int_{\Pi_T}\sigma^2(u(t,x))\beta^{\prime\prime} (u(t,x))\psi(t,x)\,dx\,dt \notag \\
  &  + \int_{\Pi_T} \int_{\mathbf{E}} \int_0^1 \eta(u(t,x);z)\beta^\prime \big(u(t,x) + \lambda\,\eta(u(t,x);z)\big)\psi(t,x)\,d\lambda\,\widetilde{N}(dz,dt)\,dx  \notag \\
 &\quad +\int_{\Pi_T} \int_{\mathbf{E}}  \int_0^1  (1-\lambda)\eta^2(u(t,x);z)\beta^{\prime\prime} \big(u(t,x) + \lambda\,\eta(u(t,x);z)\big)
 \psi(t,x)\,d\lambda\,m(dz)\,dx\,dt \notag \\
 &  \quad \ge  \int_{\Pi_T} \beta^{\prime\prime}(u(t,x)) |\grad {\tt G}(u(t,x))|^2\psi(t,x)\,dx\,dt - \int_{\R^d} \beta(u_0(x))\psi(0,x)\,dx, \quad \mathbb{P}-\text{a.s}.
 \end{align*}
 \end{itemize}
 \end{defi} 
 In \cite{BaVaWit_2014,BisMajVal-2019}, the authors have shown well-posedness of entropy solution of \eqref{eq:stoc} for  regular initial data $u_0\in L^p(\R^d)$ for $p\ge 2$.  Due to the poor regularity of the initial data $u_0\in L^1(\R^d)$, one cannot expect {\it a-priori} estimates on the  Kirchoff's function as well as on $\nabla u$ and therefore well-posedness result has to be formulated in a generalized solution framework in the sense that whenever the initial data is more regular, the solution concept must coincide with the existing entropy solution framework. One such notion of solution is known as {\em renormalized solution}. Following the deterministic paper \cite{Karlsen-2004}, we 
define the notion of  {\em renormalized stochastic entropy solution} of \eqref{eq:stoc} as follows.
\begin{defi}[Renormalized stochastic entropy solution]\label{defi:renormalized}
A $ L^1(\R^d )$-valued $\{\mathcal{F}_t: t\geq 0 \}$-predictable stochastic process $u(t)= u(t,x)$ is called renormalized stochastic entropy solution of \eqref{eq:stoc} if
 \begin{itemize}
 \item[(i)] for each $ T>0$ and $\ell >0$,
 \begin{align*}
 {\tt G}({\tt T}_{\ell}(u)) \in L^2((0,T)\times \Omega;H^1(\R^d)), \,\, \text{and} \,\, 
 \underset{0\leq t\leq T}\sup  \mathbb{E}\big[||u(t,\cdot)||_{L^1(\R^d)}\big] <+ \infty;  
 \end{align*}
 \item[(ii)]  for any $\ell >0$, and a given non-negative test function  $\psi\in C_{c}^{1,2}([0,\infty )\times\R^d) $ and  convex entropy flux triple $(\beta,\zeta,\nu)$ with $\beta^\prime(0)=0$ and $|\beta^\prime|$ bounded by some constant $K>0$, there exists a random non-negative bounded Radon measure $\mu_{\ell}^K$ on $\Pi_T$ such that
 \begin{align}
 &  \int_{\Pi_T} \Big\{ \beta({\tt T}_{\ell}(u)) \partial_t\psi(t,x) +  \nu({\tt T}_{\ell}(u))\Delta \psi(t,x) -  \grad \psi(t,x)\cdot \zeta({\tt T}_{\ell}(u)) \Big\}dx\,dt \notag \\
 & + \int_{\Pi_T} \sigma({\tt T}_{\ell}(u))\beta^\prime ({\tt T}_{\ell}(u))\psi(t,x)\,dW(t)\,dx
 + \frac{1}{2}\int_{\Pi_T}\sigma^2({\tt T}_{\ell}(u))\beta^{\prime\prime} ({\tt T}_{\ell}(u))\psi(t,x)\,dx\,dt \notag \\
  &  + \int_{\Pi_T} \int_{\mathbf{E}} \int_0^1 \eta({\tt T}_{\ell}(u);z)\beta^\prime \big({\tt T}_{\ell}(u) + \lambda\,\eta({\tt T}_{\ell}(u);z)\big)\psi(t,x)\,d\lambda\,\widetilde{N}(dz,dt)\,dx  \notag \\
 & +\int_{\Pi_T} \int_{\mathbf{E}}  \int_0^1  (1-\lambda)\eta^2({\tt T}_{\ell}(u);z)\beta^{\prime\prime} \big({\tt T}_{\ell}(u) + \lambda\,\eta({\tt T}_{\ell}(u);z)\big)
 \psi(t,x)\,d\lambda\,m(dz)\,dx\,dt \notag \\
 &  \ge  \int_{\Pi_T} \beta^{\prime\prime}({\tt T}_{\ell}(u)) |\grad {\tt G}({\tt T}_{\ell}(u))|^2\psi\,dx\,dt - \int_{\R^d} \beta({\tt T}_{\ell}(u_0))\psi(0,x)\,dx - \int_{\Pi_T} \psi\,d\mu_{\ell}^K(t,x), \quad \mathbb{P}\text{-a.s}; \label{inq:renormalized-entropy-solun}
 \end{align}
 \item[iii)] there exists a sequence $\{{\ell}_j\}$ of real numbers with ${\ell}_j \goto \infty$ as $j\goto \infty$ such that the following dissipation condition holds:
 $$ \lim_{j\goto \infty} \mathbb{E}\Big[\mu_{{\ell}_j}^K(\Pi_T)\Big]=0;$$
 \end{itemize}
 where the Lipschitz continuous truncation function ${\tt T}_{\ell}: \R \goto \R$ at height ${\ell}>0$ is given by
\begin{align}\label{eq:trucn-fn}
{\tt T}_{\ell}(u)=\begin{cases} -{\ell},\quad u<-\ell\,, \\ u,\quad |u| \le \ell\,, \\ \ell,\quad u>\ell\,.
\end{cases}
\end{align}
\end{defi}
\begin{rem}
Since for fixed $\ell >0$, ${\tt T}_{\ell}(u)\in L^\infty(\Omega\times \Pi_T)$, the terms in \eqref{inq:renormalized-entropy-solun} are all well-defined. Moreover, if $u$ is bounded renormalized entropy solution i.e., if $|u|<M$ for some $M>0$, then sending $\ell \goto \infty$ in \eqref{inq:renormalized-entropy-solun}, we see that $u$ is indeed a stochastic entropy solution of \eqref{eq:stoc} in the sense of  Definition \ref{defi:stochentropsol}.
\end{rem}
\begin{rem}\label{rem:about-initial}
A similar argument as in \cite[Remark 2.7]{BaVaWit_2012} and \cite[Lemma 2.1]{BisMaj} reveals that any renormalized entropy solution of \eqref{eq:stoc} satisfies the initial condition in the following sense: for any $\ell>0$, and any non-negative $\psi \in C_c^2(\R^d)$ with $\text{supp}(\psi) = \bar{K}$,~a compact set in $\R^d$, \begin{align*}
    \underset{h \rightarrow 0}{\lim}\,\mathbb{E}\Big[\frac{1}{h}\int_0^h\int_{\bar{K}}\big|{\tt T}_{\ell}(u(t,x)) - {\tt T}_{\ell}(u_0(x))\big|\psi(x)\,dx\,dt\Big] = 0\,.
\end{align*}
\end{rem}
The primary aim of this paper is to establish well-posedness theory for the renormalized stochastic entropy solution of  \eqref{eq:stoc}, and we do so under the following set of assumptions:
 \begin{Assumptions}
 	\item \label{A1} The initial function $u_0$ is a deterministic  integrable function i.e., $u_0\in L^1(\R^d)$.
 	\item \label{A2}  $ \Phi:\R\rightarrow \R$ is a non-decreasing  Lipschitz continuous function with $\Phi(0)=0$. 
	 \item \label{A3}  $ f=(f_1,f_2,\cdots, f_d):\R\rightarrow \R^d$ is a Lipschitz continuous function with $f_k(0)=0$, for 
 	all $1\le k\le d$.
 	\item \label{A4} We assume that $\sigma(0)=0$. Moreover, there exists positive constant $L_{\sigma} > 0$  such that 
 	\begin{align*} 
 	\big| \sigma(u)-\sigma(v)\big|  \leq L_{\sigma} |u-v|, ~\text{for all} \,\,u,v \in \R.
 	\end{align*}  
\item \label{A5}  $\eta(0,z)=0$, for all $z\in \mathbf{E}$, and there exist positive constant $L_\eta>0$ and  a function $h(z)\in L^2(\mathbf{E},m)$ with $0\le h(z)\le 1$ such that
 \begin{align*}
  \big| \eta(u;z)-\eta(v;z)\big|  \leq L_\eta |u-v|h(z),~~~\forall u,v \in \R;~~z\in \mathbf{E}\,.
 \end{align*}
 \item \label{A6}
   $\eta : \R \times \mathbf{E} \rightarrow \R$ is non-decreasing in $\R$.
\item \label{A7}
There exist positive constants $\sigma_b$ and $\eta_b$ such that $|\sigma| \le \sigma_b$ and $|\eta(\cdot~;z)| \le \eta_bh(z)$\,.
 \end{Assumptions}
 We now state the main results of this paper.
 \begin{thm}[Existence]\label{thm:existence}
 Let the assumptions \ref{A1}, \ref{A2}, \ref{A3}, \ref{A4}, \ref{A5}, \ref{A6} and \ref{A7} hold true. Then, there exists a renormalized stochastic entropy solution of \eqref{eq:stoc} in the sense of Definition \ref{defi:renormalized}. 
 \end{thm}
\begin{thm}[Uniqueness]\label{thm:uni}
Under the assumptions \ref{A1}-\ref{A6}, there exists at most one renormalized stochastic entropy solution to the Cauchy problem \eqref{eq:stoc}.
\end{thm}
\begin{rem}
The assumptions \ref{A6} and \ref{A7} are crucial for the existence of the renormalized solution. In particular, \ref{A6} is needed to handle the error terms coming from the martingale and It\^o correction terms corresponding to the jump noise; see \eqref{inq:i2j2}. The assumption \ref{A7} is only used to show the dissipation condition of the corresponding renormalized measure cf.~\eqref{inq:dissipation-measure}. We also point out that for well-posedness of entropy solution of the underlying problem with bounded and integrable initial data, assumptions \ref{A2}-\ref{A5} with $L_\eta <1$ are sufficient cf.~\cite{BisMajVal-2019}---however, since we are assuming the more stronger assumption \ref{A6}, we don't need the condition $L_\eta<1$.  
\end{rem}
\begin{rem}
The condition $\beta^{\prime}(0)=0$ in the definition of renormalized solution for \eqref{eq:stoc} is essential to manipulate the error terms coming from the martingale and It\^o correction terms corresponding to the jump noise. In particular, it is used to get the inequality \eqref{inq:i2j2} along with the assumption \ref{A6}. However, for SPDE \eqref{eq:stoc} with only Brownian noise (i.e., $\eta=0$), the condition $\beta^\prime(0)=0$ is not required. 
\end{rem}

\section{Existence of renormalized stochastic entropy solution}\label{sec:existence}
In this section, we prove Theorem \ref{thm:existence}.
We divide the proof into two cases. In the first case, we show that the unique entropy solution of the underlying problem for bounded and integrable initial data is indeed a renormalized stochastic entropy solution---which then indicates that the notion of renormalized solution is indeed a generalized solution framework. In the second case, we prove the existence results for a general integrable initial datum by using a suitable approximations method together with the analysis of the first case.
\subsection{Existence of renormalized solution for bounded and integrable initial datum} \label{subsec:1-existence}

In this subsection, we show that if $u_0\in L^\infty(\R^d) \cap L^1(\R^d)$, then any stochastic entropy solution of \eqref{eq:stoc} is indeed a renormalized stochastic entropy solution of the underlying problem. To do so, we consider the viscous approximation of  \eqref{eq:stoc}:  for $\eps>0$
 \begin{align}
 	 du_\eps(t,x) -\Delta \Phi_\eps(u_\eps(t,x))\,dt - \mbox{div}_x f(u_\eps(t,x)) \,dt&= \sigma(u_\eps(t,x))\,dW(t) + \int_{\mathbf{E}} \eta(u_\eps(t,x);z)\widetilde{N}(
 	 dz,dt),\label{eq:viscous-stoc} \\
 	 u_\eps(0,x)&=u_0(x), \notag
\end{align}
where $\Phi_\eps(x)=\Phi(x)+ \eps x$, and $u_0\in  L^\infty(\R^d) \cap L^1(\R^d)$. One can follow the argument presented in \cite{BaVaWit_2014, BisMajVal-2019} to ensure the existence of a weak solution for the problem \eqref{eq:viscous-stoc}. More precisely, we have the following proposition.
\begin{prop}\label{prop:vanishing viscosity-solution}
Let the assumptions \ref{A2}-\ref{A6} hold true and $u_0 \in L^1(\R^d) \cap L^\infty(\R^d)$. Then, for any $\eps>0$, there exists a unique weak solution $u_\eps \in N_w^2(0,T,H^1(\R^d))$  to the problem \eqref{eq:viscous-stoc}. Moreover,
$u_\eps \in L^\infty(0,T;L^2(\Omega \times\R^d))$ and there exists a constant $C>0$, independent of $\eps$, such that
\begin{align*}
	 \sup_{0\le t\le T} \mathbb{E}\Big[\big\|u_\eps(t)\big\|_{L^2(\R^d)}^2\Big]  + \eps \int_0^T \mathbb{E}\Big[\big\|\grad u_\eps(s)\big\|_{L^2(\R^d)}^2\Big]\,ds
	 + \int_0^T \mathbb{E}\Big[\big\|\grad {\tt G}(u_\eps(s))\big\|_{L^2(\R^d)}^2\Big]\,ds \le C\,.
\end{align*} 
\end{prop}
Following the proof as in \cite{BaVaWit_2014, BisMajVal-2019}, one arrives at the following theorem.
\begin{thm}\label{thm:L1-contraction-entropy}
Let the assumptions \ref{A2}-\ref{A6} hold true, $u_0 \in L^1(\R^d) \cap L^\infty(\R^d)$ and $u_\eps(\cdot)$ be the unique weak solution of \eqref{eq:viscous-stoc}.  Then, 
\begin{itemize}
\item[i)] $u_\eps$ converges to $u$
in $L^p(\Omega\times (0,T); L^p_{{\rm loc}}(\R^d))$ for $1\le p<2$, where $u$ is a unique stochastic entropy solution of \eqref{eq:stoc} in the sense of Definition \ref{defi:stochentropsol}. In particular, $u_\eps \goto u$ $\mathbb{P}$-a.s. and for a.e. $(t,x)\in \Pi_T$. 
\item[ii)] The following contraction principle holds: for a.e. $t\in (0,T)$, there exists a constant $C>0$ such that
\begin{align}
\mathbb{E}\Big[\int_{\R^d}\big| u(t,x)-v(t,x)\big|\,dx \Big] \le C \int_{\R^d} |u_0(x)-v_0(x)|\,dx\,, \label{inq:contraction-entropy-solun}
\end{align}
where $v(t,x)$ is the unique stochastic entropy solution of \eqref{eq:stoc} corresponding to the initial condition $v_0\in L^\infty(\R^d) \cap L^1(\R^d)$. 
\end{itemize}
\end{thm}
\subsubsection{\bf A-priori estimates:} We first derive some essential a-priori estimate for $u_\eps$ and related Kirchhoff's term.
For any $C^2(\R)$-function $S$, define $\zeta^S:\big( \zeta_1^S,\ldots, \zeta_d^S\big): \R \goto \R^d$ and $\nu^S:\R\goto \R$ by
$$ \zeta_i^S(r)=\int_0^r S^\prime(s) f_i^\prime(s)\,ds~~(1\le i\le d),\quad \nu^S(r)= \int_0^r \Phi^\prime(s) S^\prime(s)\,ds.$$ Then for any $0\le \psi\in C_c^\infty([0,\infty)\times \R^d)$, one has, $\mathbb{P}$-a.s.,
 \begin{align}
 &  \int_{\Pi_T} \Big\{ S(u_\eps) \partial_t\psi(t,x) +  \nu^S(u_\eps)\Delta \psi(t,x) -  \grad \psi(t,x)\cdot \zeta^S(u_\eps) + \eps \Delta S(u_\eps) \psi(t,x) \Big\}dx\,dt \notag \\
 & + \int_{\Pi_T} \sigma(u_\eps)S^\prime (u_\eps)\psi(t,x)\,dW(t)\,dx
 + \frac{1}{2}\int_{\Pi_T}\sigma^2(u_\eps)S^{\prime\prime} (u_\eps)\psi(t,x)\,dx\,dt \notag \\
  &  + \int_{\Pi_T} \int_{\mathbf{E}} \big\{S(u_\eps+\eta(u_\eps;z))-S(u_\eps)\big\}\psi(t,x)\,d\lambda\,\widetilde{N}(dz,dt)\,dx  \notag \\
 & +\int_{\Pi_T} \int_{\mathbf{E}} \big\{S(u_\eps+\eta(u_\eps;z))-S(u_\eps)-\eta(u_\eps;z)S^\prime(u_\eps)\big\}
 \psi(t,x)\,d\lambda\,m(dz)\,dx\,dt \notag \\
 & =  \int_{\Pi_T} S^{\prime\prime}(u_\eps)\big( |\grad {\tt G}(u_\eps)|^2 + \eps |\nabla u_\eps|^2 \big)\psi\,dx\,dt - \int_{\R^d} S(u_0)\psi(0,x)\,dx + \int_{\R^d} S(u_\eps(T,x))\psi(T,x)\,dx\,.\label{eq:ito-formula-viscous}
 \end{align}
By an approximations argument, \eqref{eq:ito-formula-viscous} holds for any $S\in W^{2,\infty}(\R)$.
\begin{lem}
Let $u_\eps$ be the weak solution of \eqref{eq:viscous-stoc} with initial data $u_0\in L^1\cap L^\infty(\R^d)$. Then the following {\it a-priori} estimate hold:
\begin{align}
& \sup_{0\le t\le T} \mathbb{E}\Big[\|u_\eps(t)\|_{L^1(\R^d)}\Big] \le \|u_0\|_{L^1(\R^d)}\,, \label{inq:L1}  \\
& \mathbb{E}\Big[ \int_{\Pi_T}  \Big( |\nabla {\tt G}({\tt T}_{\ell}(u_\eps(s,x)))|^2 + \eps |\nabla  {\tt T}_{\ell}(u_\eps(s,x))|^2\Big)\,dx\,ds\Big] \le C(\ell)\,, \label{inq:L1-kirchoff}
\end{align}
where $C(\ell):= \ell \|u_0\|_{L^1(\R^d)} + T \underset{0\le t\le T}\sup\, \mathbb{E}\Big[\|u_\eps(t)\|_{L^2(\R^d)}^2\Big]$.
\end{lem}
\begin{proof}
Let $0\le \phi\in C_c^\infty(\R^d)$. Taking $S(u)= \displaystyle \frac{1}{\ell} \int_0^u {\tt T}_\ell(r)\,dr$ in \eqref{eq:ito-formula-viscous} where ${\tt T}_{\ell}(\cdot)$ is defined in \eqref{eq:trucn-fn},  we have
\begin{align}
& \mathbb{E}\Big[ \int_{\R^d} \Big( \frac{1}{\ell} \int_0^{u_\eps(t,x)} {\tt T}_\ell(r)\,dr\Big)\phi(x)\,dx\Big] -  \int_{\R^d} \Big( \frac{1}{\ell} \int_0^{u_0(x)} {\tt T}_\ell(r)\,dr\Big)\phi(x)\,dx \notag \\
& \le C \mathbb{E}\Big[ \int_0^t \int_{\R^d} \big( |u_\eps(s,x)| + \eps \big) |\Delta \phi(x)|\,dx\,ds \Big] +   C \mathbb{E}\Big[ \int_0^t \int_{\R^d}  |u_\eps(s,x)|  |\nabla \phi(x)|\,dx\,ds \Big]\notag \\
& + C \mathbb{E}\Big[ \int_0^t \int_{\R^d}  u^2_\eps(s,x)S^{\prime\prime}(u_\eps(s,x)) \phi(x)\,dx\,ds\Big] \notag \\
& + C\,\mathbb{E}\Big[ \int_0^t \int_{\R^d} \int_{\mathbf{E}}  \int_0^1  (1-\lambda)u_\eps^2(s,x)S^{\prime\prime} \big(u_\eps + \lambda\,\eta(u_\eps;z)\big)
 h^2(z) \phi(x)\,d\lambda\,m(dz)\,dx\,ds\Big]\equiv \sum_{i=1}^4 \mathcal{A}_i\,. \label{esti:1-L1}
\end{align}
In the above, we have used the fact that $|S(u_\eps)| \le |u_\eps|$ and $|S^\prime(r)| \le 1$ for any $r\in \R$. 
We first estimate $\mathcal{A}_4$. Note that $S^{\prime\prime}(u)=\frac{1}{\ell} {\bf 1}_{\{|u|\le \ell\}}$---which is symmetric around zero and for any $r\in \R$, $r^2 S^{\prime\prime}(r)\le \ell$. 
Hence without loss of generality, we assume that $u_\eps(s,x)\ge 0$.  Then by the assumptions \ref{A5} and \ref{A6}, one has
$\eta(u_\eps;z) \ge \eta(0;z)=0$. Thus, 
$$ 0\le u_\eps(s,x)\le u_\eps(s,x)+\lambda \eta(u_\eps(s,x);z),\quad \forall~\lambda\in [0,1].$$
This gives
\begin{align*}
\mathcal{A}_4 & \le C \mathbb{E}\Big[ \int_0^t \int_{\R^d} \int_{\mathbf{E}}  \int_0^1  (1-\lambda) \big(u_\eps + \lambda\,\eta(u_\eps;z)\big)^2S^{\prime\prime} \big(u_\eps + \lambda\,\eta(u_\eps;z)\big)
 h^2(z) \phi(x)\,d\lambda\,m(dz)\,dx\,ds\Big] \\
 & \le C t \|\phi\|_{L^1(\R^d)} \ell\,.
\end{align*}
Since $r^2 S^{\prime\prime}(r)\le \ell$, one has $$\mathcal{A}_3\le  C t \|\phi\|_{L^1(\R^d)} \ell.$$
Combining these estimates in \eqref{esti:1-L1}, we have
\begin{align}
& \mathbb{E}\Big[ \int_{\R^d} \Big( \frac{1}{\ell} \int_0^{u_\eps(t,x)} {\tt T}_\ell(r)\,dr\Big)\phi(x)\,dx\Big] -  \int_{\R^d} \Big( \frac{1}{\ell} \int_0^{u_0(x)} {\tt T}_\ell(r)\,dr\Big)\phi(x)\,dx\notag  \\
& \le C \mathbb{E}\Big[ \int_0^t \int_{\R^d} \Big\{ \big( |u_\eps(s,x)| + \eps \big) |\Delta \phi(x)|+  |u_\eps(s,x)|  |\nabla \phi(x)|\Big\}\,dx\,ds \Big] + C 
t \|\phi\|_{L^1(\R^d)} \ell\,. \label{esti:2-L1}
\end{align}
Since $S(u_\eps)\goto |u_\eps|$ as $\ell \goto 0$, sending $\ell \goto 0$ in \eqref{esti:2-L1}, we have
\begin{align}
& \mathbb{E}\Big[ \int_{\R^d} | u_\eps(s,x)|\phi(x)\,dx\Big] - \int_{\R^d} | u_0(x)|\phi(x)\,dx \notag \\
&\le  C \mathbb{E}\Big[ \int_0^t \int_{\R^d} \Big\{ \big( |u_\eps(s,x)| + \eps \big) |\Delta \phi(x)|+  |u_\eps(s,x)|  |\nabla \phi(x)|\Big\}\,dx\,ds \Big]\,.  \label{esti:3-L1}
\end{align}
Sending $\phi\goto {\bf 1}_{\R^d}$ in  \eqref{esti:3-L1} together with dominated convergence theorem, we get \eqref{inq:L1}.

\vspace{0.5cm}

\noindent\underline{Proof of \eqref{inq:L1-kirchoff}:} Taking 
$$ S(u)=\int_0^u {\tt T}_{\ell}(r)\,dr=\begin{cases} \frac{|u|^2}{2},\quad \text{if}~~~|u|\le \ell \\
\ell |u| -\frac{\ell^2}{2}, \quad \text{if}~~~|u|> \ell\,,\end{cases} $$
in  \eqref{eq:ito-formula-viscous}, we have, for any non-negative test function $\phi \in \mathcal{D}(\R^d)$,
\begin{align}
& \mathbb{E}\Big[ \int_{\Pi_T} {\bf 1}_{\{|u_\eps|\le \ell\}} \Big( |\nabla {\tt G}(u_\eps)|^2 + \eps |\nabla u_\eps|^2\Big)\,\phi(x)\,dx\,ds\Big] \notag \\
&\le \int_{\R^d} S(u_0(x))\phi(x)\,dx + C \ell\, \mathbb{E}\Big[ \int_{\Pi_T} \big( |u_\eps(s,x)| + \eps\big) \big( |\Delta \phi(x)| + |\nabla \phi(x)|\big)\,dx\,ds\Big] \notag \\
&\qquad  + C \mathbb{E}\Big[ \int_{\Pi_T} |u_\eps(s,x)|^2\phi(x)\,dx\,ds\Big]\,. \notag
\end{align}
Replacing $\phi$ in the above inequality by the standard smooth cut-off function on $B_N(0)$ with support inside $B_{2N}(0)$ and letting $N$ goes to infinity, we arrive at
\begin{align}
& \mathbb{E}\Big[ \int_{\Pi_T} {\bf 1}_{\{|u_\eps(s,x)|\le \ell\}} \Big( |\nabla {\tt G}(u_\eps)|^2 + \eps |\nabla u_\eps|^2\Big)\,dx\,ds\Big] \notag \\
& \le \int_{\R^d} S(u_0(x))\,dx + C T \sup_{0\le t\le T} \mathbb{E}\Big[\|u_\eps(t)\|_{L^2(\R^d)}^2\Big]\,. \notag
\end{align}
Since $|S(r)| \le \ell |r|$ for any $r\in \R$, we get from the above inequality
\begin{align*}
\mathbb{E}\Big[ \int_{\Pi_T}  \Big( |\nabla {\tt G}({\tt T}_{\ell}(u_\eps(s,x)))|^2 + \eps |\nabla  {\tt T}_{\ell}(u_\eps(s,x))|^2\Big)\,dx\,ds\Big] \le C(\ell)\,,
\end{align*}
where $C(\ell):= \ell \|u_0\|_{L^1(\R^d)} + T \underset{0\le t\le T}\sup\, \mathbb{E}\Big[\|u_\eps(t)\|_{L^2(\R^d)}^2\Big]$. This completes the proof.
\end{proof}

Let us state the following technical lemma, whose proof can be found in \cite[Lemma $6$]{Blanchard-2001} under cosmetic change. 
\begin{lem}\label{lem:tech}
Let $F \in L^1(\Omega \times \Pi_T),~~F\ge 0$, and $\bar{u}, \bar{v}: \Omega \times \Pi_T \goto [0,\infty]$ be two measurable functions. Then there exists a sequence $\{{\tt s}_j\}$ of real numbers such that as $j\goto \infty$
\begin{align*}
\begin{cases}
{\tt s}_j \goto + \infty, \\
\underset{\delta \goto 0}\limsup \Bigg\{ \frac{1}{\delta} \mathbb{E}\Big[ \iint_{\{ {\tt s}_j -\delta \le \bar{u}\le {\tt s}_j +\delta \}} F\,dx\,dt\Big] +  \frac{1}{\delta} \mathbb{E}\Big[ \iint_{\{ {\tt s}_j -\delta \le \bar{v}\le {\tt s}_j +\delta \}} F\,dx\,dt\Big] \Bigg\} \goto 0.
\end{cases}
\end{align*}
\end{lem}
\begin{lem}
For any $\ell, \delta>0$, there exists a bounded function $\mathcal{E}$ on $\R^+$ such that
\begin{align}
& \mathbb{E}\Big[ \frac{1}{\delta} \iint_{\{\ell < |u_\eps|< \ell + \delta\}} \Big( |\nabla {\tt G}(u_\eps(s,x))|^2 + \eps |\nabla  u_\eps(s,x)|^2+ \frac{1}{2}\sigma^2(u_\eps(s,x)) \Big)\,dx\,ds\Big] \le \mathcal{E}(\ell)\,. \label{inq:2-radon-stoch}
\end{align}
Moreover, there exists a subsequence $\{{\ell}_j\}$ with $\ell_j\goto \infty $ as $j\goto \infty$ such that 
\begin{align}
 \lim_{j\goto \infty} \mathcal{E}(\ell_j)=0\,. \label{limit-measure-bound}
 \end{align}
\end{lem}
\begin{proof}
 For given $\ell, \delta >0$, define $S\in W^{2,\infty}$ by 
\begin{align*}
S(0)=0,\quad S^\prime(r)=\frac{1}{\delta}\big( {\tt T}_{\ell + \delta}(r)-{\tt T}_{\ell}(r)\big)=\begin{cases} -1,\quad r<-\ell-\delta, \\
\frac{r+\ell}{\delta},\quad -\ell-\delta < r < -\ell, \\0, \quad -\ell <r <\ell, \\ \frac{r-\ell}{\delta}, \quad \ell < r < \ell + \delta , \\ 1, \quad r > \ell + \delta\,.
\end{cases}
\end{align*}
Then $|S^\prime(r)| \le 1$. Hence $\zeta^S$ and $\nu^S$ are Lipschitz continuous with same Lipschitz constant of $f$ and $\Phi$ respectively. Observe that 
$$ S^{\prime\prime}(r)=\frac{1}{\delta} {\bf 1}_{\{ \ell < |r| < \ell + \delta\}}.$$
Using above $S$ in \eqref{eq:ito-formula-viscous}, we have, for any non-negative test function $\psi\in \mathcal{D}(\R^d)$,
\begin{align}
& \mathbb{E}\Big[ \frac{1}{\delta} \iint_{\{\ell < |u_\eps|< \ell + \delta\}} \Big(  |\nabla {\tt G}(u_\eps(s,x))|^2 + \eps |\nabla u_\eps(s,x)|^2+ \frac{1}{2}\sigma^2(u_\eps(s,x)) \Big)\psi(x)\,dx\,ds\Big] \notag \\
& \le \int_{\{|u_0|> \ell\}} |u_0(x)|\psi(x)\,dx + \mathbb{E}\Big[ \iint_{\{|u_\eps(s,x)| > \ell\}} \big( | u_\eps(s,x)| + \eps\big) \big( |\Delta \psi(x)| + |\nabla \psi(x)|\big)\,dx\,ds\Big] \notag \\
& \quad + \frac{1}{\delta} \mathbb{E}\Big[ \iint_{ \{\ell < |u_\eps|< \ell + \delta\}} \sigma^2(u_\eps(s,x))\psi(x)\,dx\,ds\Big] \notag \\
&  \qquad +   \mathbb{E}\Big[ \frac{1}{\delta}  \int_{|z|>0}\int_0^1 \iint_{\{\ell < |u_\eps + \lambda \eta(u_\eps;z)|< \ell + \delta\}}  (1-\lambda) \eta^2(u_\eps;z)\psi(x)\,d\lambda\,m(dz)\,dx\,ds\Big]\,. \label{inq:1-radon-stoch}
\end{align}
Invoking  ${\rm i)}$ of Theorem \ref{thm:L1-contraction-entropy},   we obtain,  after sending $\psi\goto {\bf 1}_{\R^d}$ in the resulting inequality,
\begin{align}
& \mathbb{E}\Big[ \frac{1}{\delta} \iint_{\{\ell < |u_\eps|< \ell + \delta\}} \Big( |\nabla {\tt G}(u_\eps(s,x))|^2 + \eps |\nabla  u_\eps(s,x)|^2+ \frac{1}{2}\sigma^2(u_\eps(s,x)) \Big)\,dx\,ds\Big] \notag \\
& \le \int_{\{|u_0|> \ell\}} |u_0(x)|\,dx  + \frac{1}{\delta} \mathbb{E}\Big[ \iint_{ \{\ell < |u|< \ell + \delta\}} \sigma^2(u(s,x))\,dx\,ds\Big] \notag \\
& +   \mathbb{E}\Big[ \frac{1}{\delta}  \int_{|z|>0}\int_0^1 \iint_{\{\ell < |u + \lambda \eta(u;z)|< \ell + \delta\}}  (1-\lambda) \eta^2(u;z)\,d\lambda\,m(dz)\,dx\,ds\Big] \notag \\
& \le   \int_{\{|u_0|> \ell\}} |u_0(x)|\,dx +  \limsup_{\delta \goto 0} \frac{1}{\delta} \mathbb{E}\Big[ \iint_{ \{\ell < |u|< \ell + \delta\}} \sigma^2(u(s,x))\,dx\,ds\Big] \notag \\ 
& +   \limsup_{\delta \goto 0} \frac{1}{\delta} \mathbb{E}\Big[\int_{|z|>0}\int_0^1 \iint_{\{\ell < |u + \lambda \eta(u;z)|< \ell + \delta\}}  (1-\lambda) \eta^2(u;z)\,d\lambda\,m(dz)\,dx\,ds\Big] \notag \\
& \equiv \mathcal{E}_1(\ell) +  \mathcal{E}_2(\ell) +  \mathcal{E}_3(\ell)=:\mathcal{E}(\ell)\,. \notag
\end{align}
Hence the assertion \eqref{inq:2-radon-stoch} follows if we take $\mathcal{E}(\ell):=\mathcal{E}_1(\ell) +  \mathcal{E}_2(\ell) +  \mathcal{E}_3(\ell)$. 
Since $u_0\in L^1(\R^d)\cap L^\infty(\R^d)$, $\mathcal{E}_1(\ell)\goto 0$ as $\ell \goto \infty$. Again, in view of Lemma \ref{lem:tech}, there exists a sequence $\{\ell_j\}$ of real numbers such that $\ell_j \goto \infty$ as $j\goto \infty$, and 
\begin{align*}
 \lim_{j\goto \infty} \mathcal{E}_2(\ell_j)=0,\quad  \lim_{j\goto \infty} \mathcal{E}_3(\ell_j)=0. 
 \end{align*}
 In other words, \eqref{limit-measure-bound} holds true. This completes the proof. 
\end{proof}

\subsubsection{\bf Existence proof: renormalized solution for $u_0\in L^1(\R^d)\cap L^\infty(\R^d)$}
Let $(\beta, \zeta, \nu) $ be a given convex $C^2$ entropy flux triplet with $\beta^\prime(0)=0$ and $\beta^\prime$ bounded by some constant $K>0$. 
For any $\ell, \delta>0$, define ${\tt h}_{\ell, \delta}:\R \goto \R$ by
\begin{align*}
{\tt h}_{\ell, \delta}(0)=0,\quad {\tt h}_{\ell, \delta}^\prime(r)=\begin{cases} 1,\quad |r|\le \ell, \\
\frac{\ell + \delta -|r|}{\delta},\quad \ell < |r| < \ell + \delta, \\0, \quad |r| >\ell + \delta \,.
\end{cases}
\end{align*}
Note that  ${\rm supp}\, {\tt h}_{\ell, \delta}^\prime(\cdot)\subset [-\ell -\delta, \ell + \delta]$. Moreover, ${\tt h}^{\prime\prime}_{\ell, \delta}(r)=\frac{1}{\delta}{\bf 1}_{\{\ell < |r| < \ell +\delta\}}$ and
\begin{align}
{\tt h}_{\ell, \delta}(r)\goto {\tt T}_{\ell}(r),\quad {\tt h}_{\ell, \delta}^\prime(r) \goto {\bf 1}_{\{|r| < \ell\}} \quad \text{as}~~\delta\goto 0. \label{conv:special-h}
\end{align}
Let $S=S_{\beta, {\tt h}_{\ell, \delta}}$ of the form 
\begin{align}
S_{\beta, {\tt h}_{\ell, \delta}}(0)=0,~~~S_{\beta, {\tt h}_{\ell, \delta}}^\prime(r)=\beta^\prime(r) {\tt h}_{\ell, \delta}^\prime(r)\,. \label{eq:special-fun-S}
\end{align}
Then, one has, for any smooth enough function $F: \R \rightarrow \R$
\begin{equation}\label{conv:special-S-h}
\begin{aligned}
&S_{\beta, {\tt h}_{\ell, \delta}}(r)\goto \beta({\tt T}_{\ell}(r)),~~~S_{\beta, {\tt h}_{\ell, \delta}}^\prime(r)\goto \beta^\prime({\tt T}_{\ell}(r))\,,~~~F(r)S_{\beta, {\tt h}_{\ell, \delta}}^\prime(r)\goto F({\tt T}_{\ell}(r))\beta^\prime({\tt T}_{\ell}(r))\,,\\
 &~~~\zeta^{S_{\beta, {\tt h}_{\ell, \delta}}}(r)\goto \zeta({\tt T}_{\ell}(r)),~~~ \text{and} ~~~\nu^{S_{\beta, {\tt h}_{\ell, \delta}}}(r)\goto \nu({\tt T}_{\ell}(r))\,, 
\end{aligned}
\end{equation}
as $\delta \goto 0$; see \cite{Karlsen-2004, Ibrango-2021}.
Using this $S=S_{\beta, {\tt h}_{\ell, \delta}}$ in  \eqref{eq:ito-formula-viscous}, we get, for any $0\le \psi\in \mathcal{D}([0,\infty)\times \R^d)$ and $B\in \mathcal{F}_T$,
 \begin{align}
 & \mathbb{E}\Big[ {\bf 1}_B \int_{\Pi_T} \Big\{ S_{\beta, {\tt h}_{\ell, \delta}}(u_\eps) \partial_t\psi +  \nu^{S_{\beta, {\tt h}_{\ell, \delta}}}(u_\eps)\Delta \psi -  \grad \psi\cdot \zeta^{S_{\beta, {\tt h}_{\ell, \delta}}}(u_\eps) + \eps \Delta S_{\beta, {\tt h}_{\ell, \delta}}(u_\eps) \psi \Big\}dx\,dt\Big] \notag \\
 & +\mathbb{E}\Big[ {\bf 1}_B \int_{\Pi_T} \sigma(u_\eps)S^\prime_{\beta, {\tt h}_{\ell, \delta}} (u_\eps)\psi(t,x)\,dW(t)\,dx\Big]
 + \frac{1}{2}\mathbb{E}\Big[ {\bf 1}_B\int_{\Pi_T}\sigma^2(u_\eps)\beta^{\prime\prime}(u_\eps) {\tt h}^\prime_{\ell, \delta}(u_\eps)\psi(t,x)\,dx\,dt\Big] \notag \\
 & +  \frac{1}{2\delta} \mathbb{E}\Big[ {\bf 1}_B \int_{\Pi_T}\sigma^2(u_\eps) \beta^\prime(u_\eps) {\bf 1}_{\{\ell < |u_\eps|< \ell + \delta\}}\psi(t,x)\,dx\,dt\Big] \notag \\
  &  +\mathbb{E}\Big[ {\bf 1}_B \int_{\Pi_T} \int_{\mathbf{E}} \big\{ S_{\beta, {\tt h}_{\ell, \delta}} \big(u_\eps+ \eta(u_\eps;z)\big) -  S_{\beta, {\tt h}_{\ell, \delta}} (u_\eps)\big\}\psi(t,x)\,\widetilde{N}(dz,dt)\,dx\Big]  \notag \\
 & + \mathbb{E}\Big[ {\bf 1}_B\int_{\Pi_T} \int_{\mathbf{E}}\big\{ S_{\beta, {\tt h}_{\ell, \delta}} \big(u_\eps+ \eta(u_\eps;z)\big) -  S_{\beta, {\tt h}_{\ell, \delta}} (u_\eps)-\eta(u_\eps;z)S^{\prime}_{\beta, {\tt h}_{\ell, \delta}}(u_\eps)\big\}
 \psi(t,x)\,m(dz)\,dx\,dt\Big] \notag \\
 & = \mathbb{E}\Big[ {\bf 1}_B \int_{\Pi_T} \beta^{\prime\prime}(u_\eps) {\tt h}^\prime_{\ell, \delta}(u_\eps) \big( |\grad {\tt G}(u_\eps)|^2 + \eps |\nabla u_\eps|^2 \big)\psi(t,x)\,dx\,dt\Big] \notag \\
 & + \frac{1}{\delta}\mathbb{E}\Big[ {\bf 1}_B \int_{\Pi_T} \beta^{\prime}(u_\eps){\bf 1}_{\{\ell < |u_\eps|< \ell + \delta\}}\big( |\grad {\tt G}(u_\eps)|^2 + \eps |\nabla u_\eps|^2 \big)\psi(t,x)\,dx\,dt  \Big] \notag \\
 & +\mathbb{E}\Big[ {\bf 1}_B \int_{\R^d} S_{\beta, {\tt h}_{\ell, \delta}}(u_\eps(T,x))\psi(T,x)\,dx\Big] -\mathbb{E}\Big[ {\bf 1}_B \int_{\R^d} S_{\beta, {\tt h}_{\ell, \delta}}(u_0)\psi(0,x)\,dx\Big]~\text{i.e.,}~\sum_{i=1}^6 \mathcal{H}_i=\sum_{i=7}^{10} \mathcal{H}_i\,. \label{eq:1-existence-renormalized}
\end{align}
In view of \eqref{conv:special-h} and convexity of $\beta$, we see that
\begin{align}
 \lim_{\delta\goto 0} \mathcal{H}_7
&=   \mathbb{E}\Big[ {\bf 1}_B \int_{\Pi_T} \beta^{\prime\prime}({\tt T}_{\ell}(u_\eps)) \big( |\grad {\tt G}({\tt T}_{\ell}(u_\eps))|^2 + \eps |\nabla {\tt T}_{\ell}(u_\eps)|^2 \big)\psi(t,x)\,dx\,dt\Big] \notag \\
& \ge  \mathbb{E}\Big[ {\bf 1}_B \int_{\Pi_T} \beta^{\prime\prime}({\tt T}_{\ell}(u_\eps))  |\grad {\tt G}({\tt T}_{\ell}(u_\eps))|^2 \psi(t,x)\,dx\,dt\Big]\,. \label{conv:delta-1}
\end{align}
Passing to the limit in the martingale terms in  \eqref{eq:1-existence-renormalized}~(i.e., $\mathcal{H}_2$ and $\mathcal{H}_5$) require similar reasoning as done in \cite[Subsection $4.3$]{BisMajKarl_2014}). In this regard, let $\Upsilon = \Omega \times [0,T] \times \mathbf{E},~\mathcal{M} = \mathbb{P}_T \times \mathcal{E}$ and $\pi = \mathbb{P}~ \otimes~\lambda_t~\otimes~m(dz)$, where $\mathbb{P}_T$ is the predictable $\sigma$-field on $\Omega \times [0,T]$ with respect to $\{\mathcal{F}_t\}$, and $\lambda_t$ is the Lebesgue measure on $[0,T]$. The space $L^2((\Upsilon, \mathcal{M}, \pi);\R)$ represents the square-integrable predictable integrands for It\^o-L\'evy integrals with respect to $\widetilde{N}(dz,dt)$. Thanks to the It\^o-L\'evy isometry and the martingale representation theorem, the It\^o-L\'evy integral defines an isometry between the Hilbert spaces $L^2((\Upsilon, \mathcal{M}, \pi);\R)$ and $L^2((\Omega,\mathcal{F}_T);\R)$. Indeed, if $\mathcal{J}$ denotes the It\^o-L\'evy integral operator and $\{X_n\}_n$ be a sequence converging weakly to $X$ in $L^2((\Upsilon, \mathcal{M}, \pi);\R)$, then $\mathcal{J}(X_n)$ weakly converges to $\mathcal{J}(X)$ in $L^2((\Omega,\mathcal{F}_T);\R)$. Invoking the argument of \cite[Lemma 4.5]{BisMajKarl_2014} along with \eqref{conv:special-S-h},  one has, for any $g(t,z) \in  L^2((\Upsilon, \mathcal{M}, \pi);\R)$, 
\begin{align}
    \underset{\delta \rightarrow 0}{\lim}~&\mathbb{E}\Big[\int_{\Pi_T} \int_{\mathbf{E}} \big\{S_{\beta, {\tt h}_{\ell, \delta}} \big(u_\eps+ \,\eta(u_\eps;z)\big) - S_{\beta, {\tt h}_{\ell, \delta}}(u_\eps)\big\}\psi(t,x)g(t,z)\,m(dz)\,dt\,dx\Big]\notag \\ &  = \mathbb{E}\Big[\int_{\Pi_T} \int_{\mathbf{E}}\big\{\beta\big(T_\ell(u_\eps + \eta(u_\eps;z))\big)-\beta\big(T_\ell(u_\eps)\big)\big\}\psi(t,x)g(t,z)\,m(dz)\,dt\,dx\Big].\notag
\end{align}
Thus, in view of the above discussion, we have
\begin{align}\label{eq:pass-delta-ito-levy-itegral}
    \underset{\delta \rightarrow 0}{\lim}~\mathcal{H}_5 &= \mathbb{E}\Big[{\bf 1}_B\int_{\Pi_T} \int_{\mathbf{E}} \big\{\beta\big({\tt T}_{\ell}(u_\eps+\eta(u_\eps;z))\big)-\beta({\tt T}_{\ell}(u_\eps))\big\}\psi(t,x)\,\widetilde{N}(dz,dt)\,dx\Big]\,\notag \\
    & = \mathbb{E}\Big[{\bf 1}_B\int_{\Pi_T} \int_{\mathbf{E}} \big\{\beta\big({\tt T}_{\ell}(u_\eps)+\eta(T_\ell(u_\eps);z)\big)-\beta({\tt T}_{\ell}(u_\eps))\big\}\psi(t,x)\,\widetilde{N}(dz,dt)\,dx\Big] \notag \\ &  \quad + \mathbb{E}\Big[{\bf 1}_B\int_{\Pi_T} \int_{\mathbf{E}} \big\{\beta\big({\tt T}_{\ell}(u_\eps+\eta(u_\eps;z))\big)-\beta\big({\tt T}_{\ell}(u_\eps)+\eta(T_\ell(u_\eps);z)\big)\big\}\psi(t,x)\,\widetilde{N}(dz,dt)\,dx\Big] \notag \\
    & \equiv \sum_{i=1}^{2}\mathcal{H}_{5,i}\,. 
\end{align}
Moreover, following  \cite[Lemma 4.7]{BisMaj} together with \eqref{conv:special-S-h}, we have 
\begin{align}\label{eq:pass-delta-ito-itegral}
   \underset{\delta \rightarrow 0}{\lim}~\mathcal{H}_2=
  \mathbb{E}\Big[ {\bf 1}_B \int_{\Pi_T} \sigma({\tt T}_{\ell}(u_\eps))\beta^\prime({\tt T}_{\ell}(u_\eps))\psi(t,x)\,dW(t)\,dx\Big].
\end{align}
Next we focus on the term $\mathcal{H}_6$.  Thanks to \eqref{conv:special-S-h}, one has, for any fixed $z \in \mathbf{E}$, 
\begin{align}
&\underset{\delta \rightarrow 0}{\lim}~~\mathbb{E}\Big[{\bf 1}_B\int_{\Pi_T}\big\{S_{\beta, {\tt h}_{\ell, \delta}}\big(u_\eps +\eta(u_\eps;z)\big) - S_{\beta, {\tt h}_{\ell, \delta}}(u_\eps) - \eta(u_\eps) S_{\beta, {\tt h}_{\ell, \delta}}^\beta(u_\eps)\big\}
\psi(t,x)\,dx\,dt\Big]\notag \\
 & = \mathbb{E}\Big[{\bf 1}_B\int_{\Pi_T}\big\{\beta\big(T_\ell(u_\eps +\eta(u_\eps;z))\big) - \beta(T_\ell(u_\eps)) - \eta(T_\ell(u_\eps);z)\beta^\prime(T_\ell(u_\eps))\big\}
 \psi(t,x)\,dx\,dt\Big] \notag\,, 
\end{align}
and hence using the assumption \ref{A5} and  the dominated convergence theorem, we infer that 
\begin{align}\label{eq:jump-error}
\underset{\delta \rightarrow 0}{\lim}~\mathcal{H}_6
 & = \mathbb{E}\Big[{\bf 1}_B\int_{\Pi_T} \int_{\bf E} \big\{\beta\big(T_\ell(u_\eps +\eta(u_\eps;z))\big) - \beta(T_\ell(u_\eps)) - \eta(T_\ell(u_\eps);z)\beta^\prime(T_\ell(u_\eps))\big\}
 \psi(t,x)\,m(dz)\,dx\,dt\Big]\notag \\
 & = \mathbb{E}\Big[{\bf 1}_B\int_{\Pi_T} \int_{\bf E} \big\{\beta\big(T_\ell(u_\eps) +\eta(T_\ell(u_\eps);z)\big) - \beta(T_\ell(u_\eps)) - \eta(T_\ell(u_\eps);z)\beta^\prime(T_\ell(u_\eps))\big\} \notag \\
 & \hspace{3cm} \times \psi(t,x)\,m(dz)\,dx\,dt\Big]\notag \\
 & \quad + \mathbb{E}\Big[{\bf 1}_B\int_{\Pi_T} \int_{\bf E} \big\{\beta\big(T_\ell(u_\eps +\eta(u_\eps;z))\big) - \beta\big(T_\ell(u_\eps) +\eta(T_\ell(u_\eps);z)\big)\big\}
 \psi(t,x)\,m(dz)\,dx\,dt\Big] \notag \\
 & \equiv \sum_{i=1}^2\mathcal{H}_{6,i}\,.
\end{align}

We combine $\mathcal{H}_{5,2}$ and $\mathcal{H}_{6,2}$ and use the fact that $\widetilde{N}(dz,dt)=N(dz,dt)-m(dz)dt$ to have  
\begin{align}
   & \mathcal{H}_{5,2} + \mathcal{H}_{6,2}\notag \\
   &= \mathbb{E}\Big[{\bf 1}_B\int_{\Pi_T} \int_{\mathbf{E}} \big\{\beta\big({\tt T}_{\ell}(u_\eps+\eta(u_\eps;z))\big)-\beta\big({\tt T}_{\ell}(u_\eps)+\eta(T_\ell(u_\eps);z)\big)\big\}\psi(t,x)\,N(dz,dt)\,dx\Big]\,.\label{eq:i2j2}
\end{align}
We claim that 
\begin{align}
\mathcal{H}_{5,2} + \mathcal{H}_{6,2} \le 0\,. \label{inq:i2j2}
\end{align}
We prove the inequality \eqref{inq:i2j2} in case-by-case manner. 
\vspace{0.1cm}

\underline{\textbf{Case I}~: $u_\eps > \ell$.} Since $\eta$ is non-decreasing in its first argument and $\eta(0;z) = 0$ for all $z \in \mathbf{E}$, we have 
\begin{align*}
\eta(u_\eps;z)\ge \eta(\ell,z)\ge \eta(0;z)=0,~~\text{and hence}~u_\eps + \eta(u_\eps;z) > \ell\,.
\end{align*}
Thus, using the definition of ${\tt T}_{\ell}$, we have, from \eqref{eq:i2j2}
\begin{align}
    \mathcal{H}_{5,2} + \mathcal{H}_{6,2}& = \mathbb{E}\Big[{\bf 1}_B\int_{\Pi_T} \int_{\mathbf{E}} \big\{\beta\big(\ell \big)-\beta\big(\ell +\eta(\ell;z)\big)\big\}\psi(t,x)\,N(dz,dt)\,dx\Big]\notag \\ &= - \mathbb{E}\Big[{\bf 1}_B\int_{\Pi_T} \int_{\mathbf{E}}\eta(\ell;z)\beta^\prime({\pmb r})\psi(t,x)\,N(dz,dt)\,dx\Big], \notag 
\end{align}
for some ${\pmb r} \in (\ell, \ell + \eta(\ell; z))$. Thanks to the fact that $\beta^\prime(0) =0$ and $\beta^{\prime\prime} \ge 0$, one has $\beta^\prime({\pmb r}) \ge 0$ as ${\pmb r} \in (\ell, \ell + \eta(\ell; z))$. Since the Poisson random measure $N$ is a counting measure (i.e., a positive measure), and $\psi$ is a non-negative test function, we arrive at \eqref{inq:i2j2} for $u_\eps > \ell.$
\vspace{0.1cm}

\underline{\textbf{Case II}~: $u_\eps < -\ell$.} A similar argument as invoked in {\bf Case I} yields the assertion 
\eqref{inq:i2j2} for $u_\eps <- \ell.$
\vspace{0.2cm}

\underline{\textbf{Case III}~: $0\le u_\eps \le \ell$.} It is easy to see that $ \mathcal{H}_{5,2} + \mathcal{H}_{6,2}=0$ for $u_\eps = 0$. For $0 < u_\eps \le \ell$, we re-write \eqref{eq:i2j2} as follows.
\begin{align*}
     \mathcal{H}_{5,2} + \mathcal{H}_{6,2} = \begin{cases}\displaystyle 
    \mathbb{E}\Big[{\bf 1}_B\int_{\Pi_T} \int_{\mathbf{E}} \big\{\beta\big(\ell\big)-\beta\big(u_\eps+\eta(u_\eps;z)\big)\big\}\psi(t,x)\,N(dz,dt)\,dx\Big] \quad \text{if} ~u_\eps + \eta(u_\eps;z) > \ell\,,\\
0 \qquad \text{if} ~ u_\eps + \eta(u_\eps;z) \le \ell\,, 
\end{cases}  \\
=  \begin{cases}\displaystyle \mathbb{E}\Big[{\bf 1}_B\int_{\Pi_T} \int_{\mathbf{E}} \big(\ell -u_\eps-\eta(u_\eps;z)\big)\beta^\prime({\pmb q})\psi(t,x)\,N(dz,dt)\,dx\Big] \quad \text{if} ~ u_\eps + \eta(u_\eps;z) > \ell\,,\\
0 \qquad \text{if} ~ u_\eps + \eta(u_\eps;z) \le \ell\,, 
\end{cases}
\end{align*}
 for some ${\pmb q} \in (\ell, u_\eps + \eta(u_\eps;z))$. Hence the assertion \eqref{inq:i2j2} follows once we use the fact that $\beta^\prime({\pmb q})\ge 0$ in the above expression. 
\vspace{0.2cm}

\underline{\textbf{Case IV}~: $-\ell \le u_\eps <0$.} In this case, we have
\begin{align*}
     \mathcal{H}_{5,2} + \mathcal{H}_{6,2} = \begin{cases}\displaystyle 
    \mathbb{E}\Big[{\bf 1}_B\int_{\Pi_T} \int_{\mathbf{E}} \big\{\beta\big(-\ell\big)-\beta\big(u_\eps+\eta(u_\eps;z)\big)\big\}\psi(t,x)\,N(dz,dt)\,dx\Big] \quad \text{if}~ u_\eps + \eta(u_\eps;z) <- \ell\,,\\
0 \qquad \text{if} \quad u_\eps + \eta(u_\eps;z) \ge -\ell\,, 
\end{cases}  \\
=  \begin{cases}\displaystyle \mathbb{E}\Big[{\bf 1}_B\int_{\Pi_T} \int_{\mathbf{E}} \big(-\ell -u_\eps-\eta(u_\eps;z)\big)\beta^\prime({\pmb s})\psi(t,x)\,N(dz,dt)\,dx\Big] \quad \text{if}~ u_\eps + \eta(u_\eps;z) <- \ell\,,\\
0 \qquad \text{if}~ u_\eps + \eta(u_\eps;z) \ge -\ell\,, 
\end{cases}
\end{align*}
 for some ${\pmb s} \in (u_\eps + \eta(u_\eps;z), -\ell)$. Since ${\pmb s}\le 0$ and $\beta^{\prime}(\cdot)$ is non decreasing with $\beta^\prime(0)=0$, the required inequality holds true in this case also. 

Again, using \eqref{conv:special-S-h}, we pass to the limit in $\mathcal{H}_3$ as $\delta\goto 0$ to have
\begin{align}\label{eq:ito-error}
   \underset{\delta \rightarrow 0}{\lim}~\mathcal{H}_3 = \mathbb{E}\Big[ {\bf 1}_B\int_{\Pi_T}\sigma^2(T_\ell(u_\eps))\beta^{\prime\prime}(T_\ell(u_\eps))\psi(t,x)\,dx\,dt\Big].  
\end{align}
Define the random Radon measure on $\Pi_T$: $\mathbb{P}$-a.s.,
\begin{align*}
d\mu_{\ell, \eps, \delta}^K(t,x):&=\frac{K}{\delta} {\bf 1}_{\{\ell < |u_\eps(t,x)|<\ell + \delta\}} \Big( |\grad {\tt G}(u_\eps)|^2 + \eps |\nabla u_\eps|^2 + \frac{1}{2} \sigma^2(u_\eps)\Big)\,dx\,dt\,. \notag 
\end{align*}
Thanks to \eqref{inq:2-radon-stoch}, we see that
\begin{align*}
\mathbb{E}\Big[ \mu_{\ell, \eps, \delta}^K(\Pi_T)\Big] \le \mathcal{E}(\ell)\,.
\end{align*}
As a consequence, we may assume that there exist nonnegative bounded Radon measures $\mu_{\ell, \eps}^K$ and $\mu_{\ell}^K$ on $\Pi_T$ such that, $\mathbb{P}$-a.s., 
\begin{align} \label{conv:radon-measure}
\begin{cases}
 \mu_{\ell, \eps, \delta}^K \overset{\ast}{\rightharpoonup} \mu_{\ell, \eps}^K~~\text{ in the sense of measure on $\Pi_T$ as $\delta \goto 0$}, \\
   \mu_{\ell, \eps}^K \overset{\ast}{\rightharpoonup} \mu_{\ell}^K~~\text{  in the sense of measure on $ \Pi_T$ as $\eps \goto 0$}\,.
   \end{cases}
    \end{align}
Moreover, in view of \eqref{limit-measure-bound}, there exists a subsequence $\{{\ell}_j\}$ with $\ell_j\goto \infty $ as $j\goto \infty$ such that 
\begin{align}
  \lim_{j\goto \infty} \mathbb{E}\Big[\mu_{\ell_j}^K(\Pi_T)\Big]=0\,. \label{eq:total-mass}
  \end{align}
 By using \eqref{conv:special-h} and \eqref{conv:special-S-h}, the inequality \eqref{conv:delta-1}, \eqref{eq:pass-delta-ito-levy-itegral}, \eqref{eq:pass-delta-ito-itegral},\eqref{eq:jump-error}, \eqref{inq:i2j2}, \eqref{eq:ito-error} and the convergence of Radon measure \eqref{conv:radon-measure}, we  pass to the limit as $\delta \goto 0$ in
  \eqref{eq:1-existence-renormalized} to have
  \begin{align}
 & \mathbb{E}\Big[ {\bf 1}_B \int_{\Pi_T} \beta^{\prime\prime}({\tt T}_{\ell}(u_\eps))  |\grad {\tt G}({\tt T}_{\ell}(u_\eps))|^2 \psi(t,x)\,dx\,dt\Big] \notag \\
 & \le \mathbb{E}\Big[  {\bf 1}_B \int_{\R^d} \beta({\tt T}_{\ell}(u_0))\psi(0,x)\,dx\Big] - \mathbb{E}\Big[ {\bf 1}_B \int_{\R^d} \beta({\tt T}_{\ell}(u_\eps(T,x))\psi(T,x)\,dx\Big] \notag \\
 & +   \mathbb{E}\Big[ {\bf 1}_B \int_{\Pi_T} \Big\{ \beta({\tt T}_{\ell}(u_\eps)) \partial_t\psi +  \nu({\tt T}_{\ell}(u_\eps))\Delta \psi -  \grad \psi\cdot \zeta ({\tt T}_{\ell}(u_\eps)) + \eps \beta({\tt T}_\ell(u_\eps))\Delta \psi \Big\}dx\,dt\Big] \notag \\
 & +\mathbb{E}\Big[ {\bf 1}_B \int_{\Pi_T} \sigma({\tt T}_{\ell}(u_\eps))\beta^\prime({\tt T}_{\ell}(u_\eps))\psi(t,x)\,dW(t)\,dx\Big]
 + \frac{1}{2}\mathbb{E}\Big[ {\bf 1}_B\int_{\Pi_T}\sigma^2({\tt T}_{\ell}(u_\eps))\beta^{\prime\prime}({\tt T}_{\ell}(u_\eps))\psi(t,x)\,dx\,dt\Big] \notag \\
  &  +\mathbb{E}\Big[ {\bf 1}_B \int_{\Pi_T} \int_{\mathbf{E}} \int_0^1 \eta({\tt T}_{\ell}(u_\eps);z)\beta^\prime \big( {\tt T}_{\ell}(u_\eps)+ \lambda\,\eta({\tt T}_{\ell}(u_\eps);z)\big)\psi(t,x)\,d\lambda\,\widetilde{N}(dz,dt)\,dx\Big]  \notag \\
 & + \mathbb{E}\Big[ {\bf 1}_B\int_{\Pi_T} \int_{\mathbf{E}}  \int_0^1  (1-\lambda)\eta^2({\tt T}_{\ell}(u_\eps);z)\beta^{\prime\prime} \big( {\tt T}_{\ell}(u_\eps)+ \lambda\,\eta({\tt T}_{\ell}(u_\eps);z)\big)
 \psi(t,x)\,d\lambda\,m(dz)\,dx\,dt\Big] \notag \\
 & \quad + \mathbb{E}\Big[ {\bf 1}_{B} \int_{\Pi_T} \psi(t,x)\,d\mu_{\ell, \eps}^K(t,x)\Big]\,. \label{inq:1-existence-renormalized}
  \end{align}
  Note that $u_\eps$ converges to $u$
in $L^p(\Omega\times (0,T); L^p_{{\rm loc}}(\R^d))$ for $1\le p<2$. In particular, $u_\eps \goto u$ $\mathbb{P}$-a.s., and for a.e. $(t,x)\in \Pi_T$. One can easily pass to the limit in \eqref{inq:1-existence-renormalized} (same argument as in \cite{BisMajKarl_2014}) as $\eps\goto 0$ except the first term. A similar argument as in \cite[Page $838$]{BisMajVal-2019} reveals that 
\begin{align}
& \liminf_{\eps\goto 0}  \mathbb{E}\Big[ {\bf 1}_B \int_{\Pi_T} \beta^{\prime\prime}({\tt T}_{\ell}(u_\eps))  |\grad {\tt G}({\tt T}_{\ell}(u_\eps))|^2 \psi(t,x)\,dx\,dt\Big]  \notag \\
& \ge \mathbb{E}\Big[ {\bf 1}_B \int_{\Pi_T} \beta^{\prime\prime}({\tt T}_{\ell}(u))  |\grad {\tt G}({\tt T}_{\ell}(u))|^2 \psi(t,x)\,dx\,dt\Big] \,. \label{inq:final-kirchoff-renormalized}
\end{align}
By using \eqref{conv:radon-measure} and \eqref{inq:final-kirchoff-renormalized} along with non-negativity of $\beta$, we obtain, after passing to the limit $\eps \goto 0$ in \eqref{inq:1-existence-renormalized},
\begin{align}
&  \mathbb{E}\Big[ {\bf 1}_B \int_{\Pi_T} \beta^{\prime\prime}({\tt T}_{\ell}(u))  |\grad {\tt G}({\tt T}_{\ell}(u))|^2 \psi(t,x)\,dx\,dt\Big] \notag \\
 & \le  \mathbb{E}\Big[  {\bf 1}_B \int_{\R^d} \beta({\tt T}_{\ell}(u_0))\psi(0,x)\,dx\Big]  + \mathbb{E}\Big[ {\bf 1}_{B} \int_{\Pi_T} \psi(t,x)\,d\mu_{\ell}^K(t,x)\Big] \notag \\
 & +   \mathbb{E}\Big[ {\bf 1}_B \int_{\Pi_T} \Big\{ \beta({\tt T}_{\ell}(u)) \partial_t\psi +  \nu({\tt T}_{\ell}(u))\Delta \psi -  \grad \psi\cdot \zeta ({\tt T}_{\ell}(u)) \Big\}dx\,dt\Big] \notag \\
 & +\mathbb{E}\Big[ {\bf 1}_B \int_{\Pi_T} \sigma({\tt T}_{\ell}(u))\beta^\prime({\tt T}_{\ell}(u))\psi(t,x)\,dW(t)\,dx\Big]
 + \frac{1}{2}\mathbb{E}\Big[ {\bf 1}_B\int_{\Pi_T}\sigma^2({\tt T}_{\ell}(u))\beta^{\prime\prime}({\tt T}_{\ell}(u))\psi(t,x)\,dx\,dt\Big] \notag \\
  &  +\mathbb{E}\Big[ {\bf 1}_B \int_{\Pi_T} \int_{\mathbf{E}} \int_0^1 \eta({\tt T}_{\ell}(u);z)\beta^\prime \big( {\tt T}_{\ell}(u)+ \lambda\,\eta({\tt T}_{\ell}(u);z)\big)\psi(t,x)\,d\lambda\,\widetilde{N}(dz,dt)\,dx\Big]  \notag \\
 & + \mathbb{E}\Big[ {\bf 1}_B\int_{\Pi_T} \int_{\mathbf{E}}  \int_0^1  (1-\lambda)\eta^2({\tt T}_{\ell}(u);z)\beta^{\prime\prime} \big( {\tt T}_{\ell}(u)+ \lambda\,\eta({\tt T}_{\ell}(u);z)\big)
 \psi(t,x)\,d\lambda\,m(dz)\,dx\,dt\Big] \,. \label{inq:2-existence-renormalized}
\end{align}
Again, invoking ${\rm i)}$ of Theorem \ref{thm:L1-contraction-entropy} and  the uniform moment estimates \eqref{inq:L1} and \eqref{inq:L1-kirchoff} along with the generalized Fatou's lemma, we have, for each $ T>0$ and $\ell >0$
 \begin{align*}
 {\tt G}({\tt T}_{\ell}(u)) \in L^2((0,T)\times \Omega;H^1(\R^d)), \,\, \text{and} \,\, 
 \underset{0\leq t\leq T}\sup  \mathbb{E}\big[||u(t,\cdot)||_{1}\big] <+ \infty. 
 \end{align*}
Since \eqref{inq:2-existence-renormalized} holds for any $B\in \mathcal{F}_T$, the inequality
\begin{align*}
&  \int_{\Pi_T} \beta^{\prime\prime}({\tt T}_{\ell}(u))  |\grad {\tt G}({\tt T}_{\ell}(u))|^2 \psi(t,x)\,dx\,dt -  \int_{\Pi_T} \psi(t,x)\,d\mu_{\ell}^K(t,x)\notag \\
 & \le \int_{\R^d} \beta({\tt T}_{\ell}(u_0))\psi(0,x)\,dx  +   \int_{\Pi_T} \Big\{ \beta({\tt T}_{\ell}(u)) \partial_t\psi +  \nu({\tt T}_{\ell}(u))\Delta \psi -  \grad \psi\cdot \zeta ({\tt T}_{\ell}(u)) \Big\}dx\,dt \notag \\
 & + \int_{\Pi_T} \sigma({\tt T}_{\ell}(u_\eps))\beta^\prime({\tt T}_{\ell}(u_\eps))\psi(t,x)\,dW(t)\,dx
 + \frac{1}{2}\int_{\Pi_T}\sigma^2({\tt T}_{\ell}(u))\beta^{\prime\prime}({\tt T}_{\ell}(u))\psi(t,x)\,dx\,dt\notag \\
  &  + \int_{\Pi_T} \int_{\mathbf{E}} \int_0^1 \eta({\tt T}_{\ell}(u);z)\beta^\prime \big( {\tt T}_{\ell}(u)+ \lambda\,\eta({\tt T}_{\ell}(u);z)\big)\psi(t,x)\,d\lambda\,\widetilde{N}(dz,dt)\,dx  \notag \\
 & + \int_{\Pi_T} \int_{\mathbf{E}}  \int_0^1  (1-\lambda)\eta^2({\tt T}_{\ell}(u);z)\beta^{\prime\prime} \big( {\tt T}_{\ell}(u)+ \lambda\,\eta({\tt T}_{\ell}(u);z)\big)
 \psi(t,x)\,d\lambda\,m(dz)\,dx\,dt 
\end{align*}
  holds for $\mathbb{P}$-almost surely. Together with \eqref{eq:total-mass}, we conclude that the entropy solution $u$ of \eqref{eq:stoc} for the bounded and integrable initial data $u_0$ is indeed a renormalized stochastic entropy solution of  \eqref{eq:stoc}  in the sense of Definition \ref{defi:renormalized}.
  
  \subsection{Existence of renormalized stochastic entropy solution for integrable initial datum}\label{subsec:existence-integrable-ini}
  For $n>1$, consider the following problem:
  \begin{align}\label{eq:stoch-truncation}
  \begin{cases} 
\displaystyle du_n(t,x)- \mbox{div} f(u_n) \,dt-\Delta \Phi(u_n)\,dt =
\sigma(u_n)\,dW(t) + \int_{\mathbf{E}} \eta(u_n;z)\widetilde{N}(dz,dt),~~(t,x) \in \Pi_T, \\
u_n(0,x) = {\tt T}_n(u_0(x)),~~ x\in \R^d,
\end{cases}
\end{align}
Note that, since $u_0\in L^1(\R^d)$, for each fixed $n\in \mathbb{N}$, ${\tt T}_n(u_0) \in L^\infty(\R^d)\cap L^1(\R^d)$, and therefore, equation \eqref{eq:stoch-truncation}
has a unique entropy solution $u_n$. Moreover,  in view of \eqref{inq:contraction-entropy-solun} of Theorem \ref{thm:L1-contraction-entropy}, we see that for all $t\in (0,T)$, and $n_1, n_2\in \mathbb{N},$
\begin{align*}
\mathbb{E}\Big[\|u_{n_1}(t)-u_{n_2}(t)\|_{L^1(\R^d)}\Big] \le C \| {\tt T}_{n_1}(u_0)-{\tt T}_{n_2}(u_0)\|_{L^1(\R^d)}\,.
\end{align*}
Since ${\tt T}_n(u_0)\goto u_0$ in $L^1(\R^d)$, from the above estimate, we see that $\{u_n\}$ is a Cauchy sequence in $L^1(\Omega\times \Pi_T)$. Hence there exists a $\{\mathcal{F}_t\}$-predictable process $\bar{u}$ taking values in $L^1(\R^d)$ such that
\begin{align}
u_n \goto \bar{u}~\text{ pointwise and in}~~ L^1(\Omega\times \Pi_T). \label{conv:truncation}
\end{align}
  In view of the analysis done in subsection \ref{subsec:1-existence}, each $u_n$ is a renormalized stochastic entropy solution of \eqref{eq:stoc} with initial condition ${\tt T}_n(u_0)$. Hence, for each $T>0$ and $\ell>0$
  \begin{align}
  \sup_{0\le t\le T}\mathbb{E}\Big[\|u_n(t)\|_{L^1(\R^d)}\Big] \le \|u_0\|_{L^1(\R^d)}, \quad  \nabla {\tt G}({\tt T}_\ell(u_n))\in L^2((0,T)\times \Omega; L^2(\R^d))\,. \label{esti:apriori-truncation}
  \end{align}
  Furthermore,  for any $\ell >0$, and a given non-negative test function  $\psi\in C_{c}^{1,2}([0,\infty )\times\R^d) $ and  convex entropy flux triple $(\beta,\zeta,\nu)$ with  $\beta^\prime(0)=0$ and $|\beta^\prime|$ bounded by some constant $K>0$, there holds
 \begin{align}
 &\mathbb{E}\Big[ {\bf 1}_B  \int_{\Pi_T} \Big\{ \beta({\tt T}_{\ell}(u_n)) \partial_t\psi(t,x) +  \nu({\tt T}_{\ell}(u_n))\Delta \psi(t,x) -  \grad \psi(t,x)\cdot \zeta({\tt T}_{\ell}(u_n)) \Big\}dx\,dt \Big]\notag \\
 & +\mathbb{E}\Big[ {\bf 1}_B \int_{\Pi_T} \sigma({\tt T}_{\ell}(u_n))\beta^\prime ({\tt T}_{\ell}(u_n))\psi(t,x)\,dW(t)\,dx\Big]
 + \frac{1}{2} \mathbb{E}\Big[ {\bf 1}_B \int_{\Pi_T}\sigma^2({\tt T}_{\ell}(u_n))\beta^{\prime\prime} ({\tt T}_{\ell}(u_n))\psi(t,x)\,dx\,dt\Big] \notag \\
  &  +  \mathbb{E}\Big[ {\bf 1}_B \int_{\Pi_T} \int_{\mathbf{E}} \int_0^1 \eta({\tt T}_{\ell}(u_n);z)\beta^\prime \big({\tt T}_{\ell}(u_n) + \lambda\,\eta({\tt T}_{\ell}(u_n);z)\big)\psi(t,x)\,d\lambda\,\widetilde{N}(dz,dt)\,dx\Big]  \notag \\
 & +\mathbb{E}\Big[ {\bf 1}_B \int_{\Pi_T} \int_{\mathbf{E}}  \int_0^1  (1-\lambda)\eta^2({\tt T}_{\ell}(u_n);z)\beta^{\prime\prime} \big({\tt T}_{\ell}(u_n) + \lambda\,\eta({\tt T}_{\ell}(u_n);z)\big)
 \psi(t,x)\,d\lambda\,m(dz)\,dx\,dt\Big] \notag \\
 &  \ge \mathbb{E}\Big[ {\bf 1}_B \int_{\Pi_T} \beta^{\prime\prime}({\tt T}_{\ell}(u_n)) |\grad {\tt G}({\tt T}_{\ell}(u_n))|^2\psi(t,x)\,dx\,dt\Big] -  \mathbb{E}\Big[ {\bf 1}_B \int_{\R^d} \beta({\tt T}_{\ell}({\tt T}_n(u_0)))\psi(0,x)\,dx\Big] \notag \\
 &\qquad  -  \mathbb{E}\Big[ {\bf 1}_B \int_{\Pi_T} \psi(t,x)\,d\mu_{\ell,n}^K(t,x)\Big]\,,\label{inq:renormalized-truncation}
 \end{align}
 where $\mu_{\ell, n}^K$ is the corresponding renormalized measure for $u_n$. Thanks to the assumption \ref{A4}, \ref{A5} and \ref{A7}, one has  
  \begin{align}
  \mathbb{E}\Big[ \mu_{\ell, n}^K(\Pi_T)\Big]& \le  K\int_{\{|{\tt T}_n(u_0)|> \ell\}} |{\tt T}_n(u_0(x))|\,dx +  \limsup_{\delta \goto 0} \frac{K}{\delta} \mathbb{E}\Big[ \iint_{ \{\ell < |u_n|< \ell + \delta\}} \sigma^2(u_n)\,dx\,ds\Big] \notag \\ 
& +   \limsup_{\delta \goto 0} \frac{K}{\delta} \mathbb{E}\Big[\int_{|z|>0}\int_0^1 \iint_{\{\ell < |u_n + \lambda \eta(u_n;z)|< \ell + \delta\}}  (1-\lambda) \eta^2(u_n;z)\,d\lambda\,m(dz)\,dx\,ds\Big] \notag \\
& \le K\int_{\{|{\tt T}_n(u_0)|> \ell\}} |{\tt T}_n(u_0(x))|\,dx +  \limsup_{\delta \goto 0} \frac{K\sigma_b\,L_\sigma}{\delta} \mathbb{E}\Big[ \iint_{ \{\ell < |u_n|< \ell + \delta\}}|u_n|\,dx\,ds\Big] \notag \\ 
& +   \limsup_{\delta \goto 0} \frac{K\eta_b\, L_\eta}{\delta} \mathbb{E}\Big[\int_{|z|>0}\int_0^1 \iint_{\{\ell < |u_n + \lambda \eta(u_n;z)|< \ell + \delta\}}|u_n||h(z)|^2\,d\lambda\,m(dz)\,dx\,ds\Big] \notag \\
& \le K\int_{\{|(u_0)|> \ell\}} |u_0(x)|\,dx +  \limsup_{\delta \goto 0} \frac{K\sigma_b\,L_\sigma}{\delta} \mathbb{E}\Big[ \iint_{ \{\ell < |\bar{u}|< \ell + \delta\}}|\bar{u}|\,dx\,ds\Big] \notag \\ 
& +   \limsup_{\delta \goto 0} \frac{K\eta_b\, L_\eta}{\delta} \mathbb{E}\Big[\int_{|z|>0}\int_0^1 \iint_{\{\ell < |\bar{u} + \lambda \eta(\bar{u};z)|< \ell + \delta\}}|\bar{u}||h(z)|^2\,d\lambda\,m(dz)\,dx\,ds\Big]:=\bar{\mathcal{E}}(\ell)\,, \label{inq:dissipation-measure}
  \end{align}
  where in the last inequality, we have used \eqref{conv:truncation}.  Hence there exists a random nonnegative bounded Radon measure $\bar{\mu}_{\ell}^K$ on $\Pi_T$ such that, $\mathbb{P}$-a.s.,  
  \begin{align} \label{conv:radon-measure-general-}
 \mu_{\ell, n}^K \overset{\ast}{\rightharpoonup} \bar{\mu}_{\ell}^K~~\text{ in the sense of measure on $\Pi_T$ as $n \goto \infty.$}
    \end{align}
Moreover, thanks to \cite[Lemma $6$]{Blanchard-2001}, there exists a subsequence $\{{\ell}_k\}$ with $\ell_k\goto \infty $ as $k\goto \infty$ such that 
\begin{align}
  \lim_{k\goto \infty} \mathbb{E}\Big[\bar{\mu}_{\ell_k}^K(\Pi_T)\Big]=0\,. \label{eq:total-mass-general}
  \end{align}
 Equipped with \eqref{conv:radon-measure-general-}, \eqref{eq:total-mass-general}, \eqref{esti:apriori-truncation} and the strong convergence of $u_n$ to $\bar{u}$, one can send the limit as $n\goto \infty$  in \eqref{inq:renormalized-truncation}~(similar argument as done for the viscous approximations $\{u_\eps\}$) to conclude that the strong limit $\bar{u}$ of the sequence $\{u_n\}$ is indeed a renormalized stochastic entropy solution of the problem \eqref{eq:stoc}.

  \section{Uniqueness of renormalized  stochastic entropy solution}
  This section is devoted to prove Theorem \ref{thm:uni}. We first compare any renormalized solution and the viscous solution associated to the bounded and integrable initial datum based on stochastic version of Kruzkov's doubling the variables technique to derive a Kato type inequality and then using the Kato inequality for entropy solution $u_n$ as introduced in Subsection \ref{subsec:existence-integrable-ini}, we prove that any renormalized stochastic entropy solution of the underlying problem is the  $L^1$-limit of the sequence $\{u_n\}$.
  \subsection{Derivation of Kato type inequality}
  We  first introduce the following special  function. Let $\beta:\R \rightarrow \R$ be a $C^\infty$ function satisfying 
    \begin{align*}
          \beta(0) = 0,~~ \beta(-r)= \beta(r),~~ \beta^\prime(-r) = -\beta^\prime(r),~~ \beta^{\prime\prime} \ge 0, \quad 
    \beta^\prime(r)=
    \begin{cases} 
    -1,\quad &\text{if} ~ r\le -1,\\
    \in [-1,1], \quad &\text{if}~ |r|<1,\\
    +1, \quad &\text{if} ~ r\ge 1.
    \end{cases}
    \end{align*}
     For any $\xi > 0$, define  $\beta_\xi:\R \rightarrow \R$ by $
             \beta_\xi(r) := \xi \beta(\frac{r}{\xi})$. Then
    \begin{align*}
     |r|-M_1\xi \le \beta_\xi(r) \le |r|\quad \text{and} \quad |\beta_\xi^{\prime\prime}(r)| \le \frac{M_2}{\xi} {\bf 1}_{|r|\le \xi},
    \end{align*} 
    where $ M_1 := \sup_{|r|\le 1}\big | |r|-\beta(r)\big |,~~M_2 := \sup_{|r|\le 1}|\beta^{\prime\prime} (r)|$.  For $\beta=\beta_\xi$, we define 
\begin{align}
   & F_k^{\beta}(a,b):= \int_b^a \beta'(r-b)f'_k(r)dr, \quad F_k(a,b):= sign(a-b)(f_k(a) - f_k(b))\quad (1\le k\le d), \notag \\
    & F(a,b):= (F_1(a,b), F_2(a,b), ..., F_d(a,b)), \quad \Phi^\beta(a,b):=\int_b^a \beta^\prime(r-b)\Phi^\prime(r)\,dr\,. \notag
\end{align}
Since the limit process of the viscous solution is the unique entropy solution of the underlying problem,  it is crucial to compare any renormalized entropy solution to with viscous solution $v_\eps$ of \eqref{eq:viscous-stoc} with initial condition $v_0\in L^\infty\cap L^1(\R^d)$. Note that $v_\eps \in H^1(\R^d)$ but for our analysis, we need $H^2(\R^d)$-regularity of $v_\eps$.  Following \cite{BaVaWit_2014, BisMajVal-2019},  we need to regularize $v_\eps$ by convolution.  Let $\{\tau_\kappa\}$ be a sequence  of mollifiers in $\R^d$. 
  Since $v_\eps$ is a weak solution of \eqref{eq:viscous-stoc} with initial condition $v_0\in L^\infty\cap L^1(\R^d)$,  $v_\eps \ast \tau_\kappa$ is a solution to the problem 
  
  \begin{align}
&  ( v_\eps \ast \tau_\kappa) -\int_0^t \Delta (\Phi(v_\eps)\ast \tau_\kappa)\,ds= \int_0^t {\rm div}(f(v_\eps)\ast \tau_\kappa)\,ds + \int_0^t \big( \sigma(v_\eps)\ast \tau_\kappa\big)\,dW(s) \notag \\
 & + \int_{\mathbf{E}}\int_0^t \big( \eta(v_\eps;z)\ast \tau_\kappa\big) \widetilde{N}(dz,ds) + \eps \int_0^t \Delta(v_\eps \ast \tau_\kappa)\,ds\quad \text{for a.e.}~t>0,~x\in \R^d\,. \label{eq:viscous-regular}
  \end{align}
  As usual, we use Kru\v{z}kov's technique of doubling the variables to get Kato type inequality. To proceed further, we define the test function in variables $(t,x,s,y)$ as follows. Let $\rho$ and $\varrho$ be the  non-negative, standard mollifiers on $\mathbb{R}$ and $\mathbb{R}^d$ respectively such that $supp(\rho)$ $\subset$ $[-1,0]$ and $supp (\varrho) = $ $\bar{B}_1(0)$, where $\bar{B}_1(0)$ is the closed unit ball. We define
\[\rho_{\delta_0}(r):= \frac{1}{\delta_0}\rho(\frac{r}{\delta_0}), \hspace{3mm} \varrho_\delta(x):= \frac{1}{\delta^d}\varrho(\frac{x}{\delta}), \quad \varphi_{\delta_0, \delta}(t,x,s,y):= \rho_{\delta_0}(t-s)\varrho_\delta(x-y)\psi(t,x),\]
where $\delta, \delta_0 >0$ are constants and $ 0\le \psi \in C_c^{1,2}([0, \infty) \times \mathbb{R}^d)$ is a test function.   Let $0\le J $ be the standard symmetric  mollifier on $\mathbb{R}$ with support in $[-1,1]$. For $l>0$, define $J_l(r): = \frac{1}{l}J(\frac{r}{l})$.

   Let $u(t,x)$ be a renormalized stochastic entropy solution of \eqref{eq:stoc} with initial condition $u_0(x)$.  We multiply the renormalized entropy inequality \eqref{inq:renormalized-entropy-solun} against the convex entropy flux triplet $\big( \beta(\cdot-c), F^{\beta}(\cdot,c), \Phi^{\beta}(\cdot,c)$ and the test function $\varphi_{\delta_0, \delta}(t,x,s,y)$, by $J_{l}(v_\eps\ast \tau_\kappa(s,y)-c)$ for $c\in \R$, and then integrate with respect to $s, y$ and $c$. Taking  expectation in the resulting expressions, we get
  \begin{align}
 &  \mathbb{E}\Big[ \int_{\Pi_T} \int_{\R}\int_{\Pi_T} \beta({\tt T}_{\ell}(u)-c) \partial_t \varphi_{\delta_0, \delta}(t,x,s,y) J_{l}(v_\eps\ast \tau_\kappa(s,y)-c)\,dx\,dt\,dy\,ds\,dc\Big] \notag \\
 & + \mathbb{E}\Big[ \int_{\Pi_T} \int_{\R}\int_{\Pi_T}  \Phi^\beta({\tt T}_{\ell}(u),c)\Delta_x  \varphi_{\delta_0, \delta}  J_{l}(v_\eps\ast \tau_\kappa(s,y)-c)\,dx\,dt\,dy\,ds\,dc\Big] \notag \\
& -   \mathbb{E}\Big[ \int_{\Pi_T} \int_{\R}\int_{\Pi_T}  \grad_x  \varphi_{\delta_0, \delta}(t,x,s,y) \cdot F^\beta({\tt T}_{\ell}(u),c) J_{l}(v_\eps\ast \tau_\kappa(s,y)-c)\,dx\,dt\,dy\,ds\,dc\Big] \notag \\
 & +  \mathbb{E}\Big[ \int_{\Pi_T} \int_{\R} \int_{\Pi_T} \sigma({\tt T}_{\ell}(u))\beta^\prime ({\tt T}_{\ell}(u)-c)\varphi_{\delta_0, \delta}(t,x,s,y) \,dW(t)\,dx J_{l}(v_\eps\ast \tau_\kappa(s,y)-c)\,dy\,ds\,dc\Big] \notag \\
 & + \frac{1}{2} \mathbb{E}\Big[ \int_{\Pi_T} \int_{\R} \int_{\Pi_T}\sigma^2({\tt T}_{\ell}(u))\beta^{\prime\prime} ({\tt T}_{\ell}(u)-k)\varphi_{\delta_0, \delta}(t,x,s,y) \,dx\,dt J_{l}(v_\eps\ast \tau_\kappa(s,y)-c)\,dy\,ds\,dc\Big] \notag \\
  &  + \mathbb{E}\Big[ \int_{\Pi_T} \int_{\R}  \int_{\Pi_T} \int_{\mathbf{E}} \int_0^1 \eta({\tt T}_{\ell}(u);z)\beta^\prime \big({\tt T}_{\ell}(u) + \lambda\,\eta({\tt T}_{\ell}(u);z)-c\big)\varphi_{\delta_0, \delta}(t,x,s,y)\,d\lambda\,\widetilde{N}(dz,dt)\,dx\notag \\
  & \hspace{4cm} \times   J_{l}(v_\eps\ast \tau_\kappa(s,y)-c)\,dy\,ds\,dc\Big] \notag \\
 & + \mathbb{E}\Big[ \int_{\Pi_T} \int_{\R} \int_{\Pi_T} \int_{\mathbf{E}}  \int_0^1  (1-\lambda)\eta^2({\tt T}_{\ell}(u);z)\beta^{\prime\prime} \big({\tt T}_{\ell}(u) + \lambda\,\eta({\tt T}_{\ell}(u);z)-c\big) \varphi_{\delta_0, \delta}(t,x,s,y)
 \,d\lambda\,m(dz)\,dx\,dt\notag \\
 & \hspace{4cm} \times  J_{l}(v_\eps\ast \tau_\kappa(s,y)-c)\,dy\,ds\,dc\Big] \notag \\
 &  \ge  \mathbb{E}\Big[ \int_{\Pi_T} \int_{\R} \int_{\Pi_T} \beta^{\prime\prime}({\tt T}_{\ell}(u)-c) |\grad {\tt G}({\tt T}_{\ell}(u))|^2 \varphi_{\delta_0, \delta}(t,x,s,y)\,dx\,dt \,J_{l}(v_\eps\ast \tau_\kappa(s,y)-c)\,dy\,ds\,dc\Big] \notag \\
 &  -  \mathbb{E}\Big[ \int_{\Pi_T} \int_{\R} \int_{\R^d} \beta({\tt T}_{\ell}(u_0)-c)\,\varphi_{\delta_0, \delta}(0,x,s,y)\,dx \,J_{l}(u_\eps\ast \tau_\kappa(s,y)-c)\,dy\,ds\,dc\Big] \notag \\
 &   - \mathbb{E}\Big[ \int_{\Pi_T} \int_{\R}   \int_{\Pi_T} \varphi_{\delta_0, \delta}(t,x,s,y) \,d\mu_{\ell}^u(t,x)\,J_{l}(v_\eps\ast \tau_\kappa(s,y)-c)\,dy\,ds\,dc\Big]\,, \label{inq:1-uni}
 \end{align}
 where $\mu_{\ell}^u$ is the corresponding renormalized measure for $u$. There exists a sequence $\{\ell_i\}$ with $\ell_{i}\goto \infty$ as $i\goto \infty$ such that 
 \begin{align}
 \lim_{i\goto \infty} \mathbb{E}\Big[ \mu_{\ell_i}^u (\Pi_T)\Big]=0\,. \label{lim:renormalized-measure-uni}
 \end{align}

 We apply It\^o-L\'evy formula to \eqref{eq:viscous-regular},  multiply by $J_l({\tt T}_{\ell}(u(t,x)) - c)$ and then integrate with respect to $t, x$ and $c$ and take expectation in the resulting expression. The result is 
 \begin{align}
 & \eps \mathbb{E}\Big[  \int_{\Pi_T} \int_{\R} \int_{\R^d} \beta^{\prime\prime}(u_\eps \ast \tau_\kappa-c) |\nabla(u_\eps \ast \tau_\kappa)|^2\varphi_{\delta_0, \delta}(t,x,s,y)\,ds\,dy \, J_l({\tt T}_{\ell}(u(t,x)) - c)\,dx\,dt\,dc\Big] \notag \\
 & + \mathbb{E}\Big[  \int_{\Pi_T} \int_{\R} \int_{\R^d} \beta^{\prime\prime}(u_\eps \ast \tau_\kappa-c) \nabla( \Phi(u_\eps) \ast \tau_\kappa)\cdot \nabla (u_\eps \ast \tau_\kappa)\varphi_{\delta_0, \delta}\,ds\,dy \, J_l({\tt T}_{\ell}(u(t,x)) - c)\,dx\,dt\,dc\Big] \notag \\
 & \le \mathbb{E}\Big[  \int_{\Pi_T} \int_{\R} \int_{\R^d} \beta(u_\eps \ast \tau_\kappa(0,y)-c)\,\varphi_{\delta_0, \delta}(t,x,0,y)\,dy J_l({\tt T}_{\ell}(u(t,x)) - c)\,dx\,dt\,dc\Big] \notag \\
 & +  \mathbb{E}\Big[  \int_{\Pi_T} \int_{\R} \int_{\Pi_T} \beta(u_\eps \ast \tau_\kappa(s,y)-c) \partial_s \varphi_{\delta_0, \delta}(t,x,s,y)\,ds\,dy \, J_l({\tt T}_{\ell}(u(t,x)) - c)\,dx\,dt\,dc\Big] \notag \\
 & -  \mathbb{E}\Big[  \int_{\Pi_T} \int_{\R} \int_{\Pi_T} \beta^{\prime}(u_\eps \ast \tau_\kappa-c) (f(u_\eps)\ast \tau_\kappa)\cdot \nabla_y \varphi_{\delta_0, \delta}(t,x,s,y)\,ds\,dy \, J_l({\tt T}_{\ell}(u(t,x)) - c)\,dx\,dt\,dc\Big] \notag \\
 & -  \mathbb{E}\Big[  \int_{\Pi_T} \int_{\R} \int_{\Pi_T} \beta^{\prime\prime}(u_\eps \ast \tau_\kappa-c) (f(u_\eps)\ast \tau_\kappa)\cdot \nabla_y(u_\eps \ast \tau_\kappa) \varphi_{\delta_0, \delta}\,ds\,dy \, J_l({\tt T}_{\ell}(u(t,x)) - c)\,dx\,dt\,dc\Big] \notag \\
 &-  \mathbb{E}\Big[  \int_{\Pi_T} \int_{\R} \int_{\Pi_T} \beta^{\prime}(u_\eps \ast \tau_\kappa-c)\nabla(\Phi(u_\eps)\ast \tau_\kappa)\cdot \nabla_y \varphi_{\delta_0, \delta}\,ds\,dy \, J_l({\tt T}_{\ell}(u(t,x)) - c)\,dx\,dt\,dc\Big] \notag \\
 & - \eps  \mathbb{E}\Big[  \int_{\Pi_T} \int_{\R} \int_{\Pi_T} \beta^{\prime}(u_\eps \ast \tau_\kappa-c) \nabla_y (u_\eps \ast \tau_\kappa)\cdot \nabla_y \varphi_{\delta_0, \delta}\,ds\,dy \, J_l({\tt T}_{\ell}(u(t,x)) - c)\,dx\,dt\,dc\Big] \notag \\
 & + \mathbb{E}\Big[  \int_{\Pi_T} \int_{\R} \int_{\Pi_T}  \sigma(u_\eps \ast \tau_\kappa) \beta^{\prime}(u_\eps \ast \tau_\kappa-c) \varphi_{\delta_0, \delta}\,dW(s)\,dy\, \, J_l({\tt T}_{\ell}(u(t,x)) - c)\,dx\,dt\,dc\Big] \notag \\
 & + \frac{1}{2}  \mathbb{E}\Big[  \int_{\Pi_T} \int_{\R} \int_{\Pi_T}  \sigma^2(u_\eps \ast \tau_\kappa) \beta^{\prime\prime}(u_\eps \ast \tau_\kappa-c) \varphi_{\delta_0, \delta}\,ds\,dy \, J_l({\tt T}_{\ell}(u(t,x)) - c)\,dx\,dt\,dc\Big] \notag \\
 &  + \mathbb{E}\Big[  \int_{\Pi_T} \int_{\R} \int_{\Pi_T} \int_{\mathbf{E}}\int_0^1 (\eta(u_\eps;z)\ast \tau_\kappa) \beta^{\prime}(u_\eps \ast \tau_\kappa + \lambda(\eta(u_\eps;z)\ast \tau_\kappa)  -c) \notag \\
 & \hspace{4cm} \times  \varphi_{\delta_0, \delta} \widetilde{N}(dz,ds)\,d\lambda \,dy  \, J_l({\tt T}_{\ell}(u(t,x)) - c)\,dx\,dt\,dc\Big] \notag \\
 & + \mathbb{E}\Big[  \int_{\Pi_T} \int_{\R} \int_{\Pi_T} \int_{\mathbf{E}}\int_0^1(1-\lambda) (\eta(u_\eps;z)\ast \tau_\kappa)^2 \beta^{\prime\prime}(u_\eps \ast \tau_\kappa + \lambda(\eta(u_\eps;z)\ast \tau_\kappa)  -c) \notag \\
 & \hspace{4cm} \times  \varphi_{\delta_0, \delta} \,m(dz)\,ds\,d\lambda \,dy  \, J_l({\tt T}_{\ell}(u(t,x)) - c)\,dx\,dt\,dc\Big] \,. \label{inq:2-uni}
 \end{align}
 We now add \eqref{inq:1-uni} and \eqref{inq:2-uni} and look for the passage to the limit with respect to the different approximation
parameters. Note that for fixed $\ell>0$, ${\tt T}_{\ell}(u)\in L^\infty(\Omega\times \Pi_T)$. Moreover, we can follow the same line of arguments as invoked in 
\cite{BisMajKarl_2014,BaVaWit_2014,BisMajVal-2019} together with Remark \ref{rem:about-initial} to arrive at the following Kato type inequality: for any  nonnegative function $\phi$ with compact support such that $\phi\in H^1([0,\infty)\times \R^d)$, there holds
\begin{align}
0& \le \mathbb{E}\Big[\int_{\R^d} | v_0(x)-{\tt T}_{\ell}(u_0(x))|\, \phi(0,x)\,dx\Big] + \mathbb{E}\Big[ \int_{\Pi_T} | v(t,x)- {\tt T}_{\ell}(u(t,x))|\partial_t \phi(t,x)\,dx\,dt\Big]\notag \\
& - \mathbb{E}\Big[ \int_{\Pi_T}  F\big( v(t,x), {\tt T}_{\ell}(u(t,x))\big)\cdot \nabla \phi(t,x)\,dx\,dt\Big] + \mathbb{E}\Big[ \int_{\Pi_T} \phi(t,x)\,d\mu_{\ell}^u(t,x)\Big]
 \notag \\
& -  \mathbb{E}\Big[ \int_{\Pi_T}  \nabla \big( \Phi(v(t,x))-\Phi({\tt T}_{\ell}(u(t,x)))\big) \cdot \nabla \phi(t,x)\,dx\,dt\Big]\,, \label{inq:3-uni}
\end{align}
where $v(t,x)$, the strong limit of the sequence $\{v_\eps\}$, is a stochastic entropy solution of \eqref{eq:stoc} with initial condition $v_0\in L^1(\R^d)\cap L^\infty(\R^d)$.

\subsection{Proof of Theorem \ref{thm:uni}}
Let $u(t,x)$ be a renormalized stochastic entropy solution of \eqref{eq:stoc} and $u_n$ be the unique entropy solution of  \eqref{eq:stoch-truncation} with initial condition ${\tt T}_{\ell}(u_0)\in L^1(\R^d)\cap L^\infty(\R^d)$. Then, in view of  {\bf Kato type inequality} \eqref{inq:3-uni}, we have, for any nonnegative function  $\phi\in H^1([0,\infty)\times \R^d)$ with compact support, 
\begin{align}
0& \le \mathbb{E}\Big[\int_{\R^d} | {\tt T}_{n}(u_0(x))-{\tt T}_{\ell}(u_0(x))|\, \phi(0,x)\,dx\Big] + \mathbb{E}\Big[ \int_{\Pi_T} | u_n(t,x)- {\tt T}_{\ell}(u(t,x))|\partial_t \phi(t,x)\,dx\,dt\Big]\notag \\
& - \mathbb{E}\Big[ \int_{\Pi_T}  F\big( u_n(t,x), {\tt T}_{\ell}(u(t,x))\big)\cdot \nabla \phi(t,x)\,dx\,dt\Big] + \mathbb{E}\Big[ \int_{\Pi_T} \phi(t,x)\,d\mu_{\ell}^u(t,x)\Big]
 \notag \\
& -  \mathbb{E}\Big[ \int_{\Pi_T}  \nabla \big( \Phi(u_n(t,x))-\Phi({\tt T}_{\ell}(u(t,x)))\big) \cdot \nabla \phi(t,x)\,dx\,dt\Big]\,. \label{inq:4-uni}
\end{align}
Note that \eqref{inq:4-uni} holds with $\phi(s,x)=\phi_m(x)\phi_h^t(s)$ for every $m\ge 1$, where
\begin{align*}
\phi_m(x)=\begin{cases} 1,\quad \text{if}~~~|x|<m, \\ \frac{m^a}{|x|^a},\quad \text{if}~~~|x|>m, \end{cases}
\end{align*}
where $a:=\frac{d}{2} + \theta$ for some $\theta>0$ such that $\phi_m\in L^2(\R^d)$, and for each $h>0$ and fixed $t\ge 0$,
\begin{align*}
\phi_h^t(s)=\begin{cases} 1,\quad \text{if}~~~s\le t\,, \\ \frac{h-s+t}{h}, \quad \text{if}~~~t\le s\le t+h\,, \\0, \quad \text{if}~~~s>t+h\,. \end{cases}
\end{align*}
Following \cite[pages $836$-$837$]{BisMajVal-2019}, we get , for a.e. $t\ge 0$
\begin{align}
& \mathbb{E}\Big[ \int_{\R^d} \big| u_n(t,x)-{\tt T}_{\ell}(u(t,x))\big| \phi_m(x)\,dx\Big] -\mathbb{E}\Big[ \int_0^t \int_{\R^d} d\mu_{\ell}^u(s,x)\Big]\notag \\
& \le C \mathbb{E}\Big[ \int_0^t \int_{\R^d} 
\big| u_n(s,x)-{\tt T}_{\ell}(u(s,x))\big| \phi_m(x)\,dx\,ds\Big] + \mathbb{E}\Big[\int_{\R^d}\big| {\tt T}_{n}(u_0(x))-{\tt T}_{\ell}(u_0(x))\big|\, \phi_m(x)\,dx\Big]\,.  \label{inq:5-uni}
\end{align}
Keeping \eqref{lim:renormalized-measure-uni} in mind, we send $\ell \goto \infty$~(along the subsequence) and then $m\goto \infty$ in \eqref{inq:5-uni} to have for a.e. $t\ge 0$,
\begin{align}
& \mathbb{E}\Big[ \int_{\R^d} \big| u_n(t,x)-u(t,x)\big|\,dx\Big] 
 \le C \mathbb{E}\Big[ \int_0^t \int_{\R^d} 
\big| u_n(s,x)-u(s,x)\big|\,dx\,ds\Big] + \mathbb{E}\Big[\int_{\R^d}\big| {\tt T}_{n}(u_0)-u_0\big|\,dx\Big]\,.  \label{inq:6-uni}
\end{align}
  An application of weaker version of Gronwall's  inequality together with the fact that $\{u_n\}$ converges in $L^1(\Omega\times \Pi_T)$ to $\bar{u}$, a renormalized stochastic entropy solution of \eqref{eq:stoc} and ${\tt T}_n(u_0)$ converges to $u_0$ in $L^1(\R^d)$, we get from \eqref{inq:6-uni}, $\mathbb{P}$-a.s.,
  $$ \bar{u}(t,x)=u(t,x),\quad \text{a.e. $(t,x)\in \Pi_T$}.$$
  In other words, renormalized stochastic entropy solution of \eqref{eq:stoc} is unique. This completes the proof of Theorem \ref{thm:uni}.

  \vspace{0.5cm}

 \noindent{\bf Acknowledgement:} 
The first author would like to acknowledge the financial support by CSIR, India.

\vspace{.2cm}

\noindent{\bf Data availability:}  Data sharing does not apply to this article as no data sets were generated or analyzed during the current study.

\vspace{0.2cm}

\noindent{\bf Conflict of interest:}  The authors have not disclosed any competing interests.

 \vspace{0.5cm}


\end{document}